\documentclass[final]{dmtcs-episciences}

\usepackage[shortlabels]{enumitem}
\usepackage[T1]{fontenc}
\usepackage[sc]{mathpazo}
\usepackage[utf8]{inputenc}

\usepackage{amsmath}
\usepackage{amssymb}
\usepackage{amsthm}
\usepackage{array}

\theoremstyle{plain}
\newtheorem{theorem}{Theorem}

\newtheorem{proposition}[theorem]{Proposition}
\newtheorem{observation}[theorem]{Observation}
\newtheorem{conjecture}[theorem]{Conjecture}

\theoremstyle{definition}
\newtheorem{definition}[theorem]{Definition}

\theoremstyle{remark}

\usepackage{bbding} 

\usepackage{microtype}
\usepackage{cite}
\def\internallinenumbers{} 
\usepackage{etoolbox}
\makeatletter\pretocmd{\section}{\addtocontents{toc}{\vspace{-2.3mm}}}{}{}\makeatother

\usepackage{mathtools}
\usepackage{xcolor}
\usepackage{hyperref}
\definecolor{darkgreen}{rgb}{0,0.4,0}
\definecolor{BrickRed}{rgb}{0.65,0.08,0}
\usepackage{hyperref}\hypersetup{colorlinks=true,linkcolor=blue,citecolor=darkgreen,filecolor=BrickRed,urlcolor=darkgreen}
\newcommand{\U}{\mathsf{U}}
\newcommand{\D}{\mathsf{D}}
\newcommand{\xs}{\mathsf{S}}
\newcommand{\xl}{\mathsf{L}}
\newcommand{\xa}{\mathsf{a}}
\newcommand{\xd}{\mathsf{d}}

\renewcommand{\P}{\mathcal{P}}
\newcommand{\im}{\mathrm{Im}(T)}
\newcommand{\imt}[1]{\mathrm{Im}(T^{#1})}
\newcommand{\imtm}{\imt{m}}
\newcommand{\sh}{\mathsf{shadow}}
\newcommand{\cost}{\mathsf{cost}}
\newcommand{\di}{\mathsf{d}}
\newcommand{\x}{\mathcal{SL}} 
\newcommand{\LI}{\mathcal{LI}}
\newcommand{\id}{\mathsf{id}}
\renewcommand{\th}{\mathsf{thin}}
\newcommand{\li}{\preccurlyeq}
\DeclarePairedDelimiter\abs{\lvert}{\rvert}
\newcommand{\oeis}[1]{\text{\href{https://oeis.org/#1}{{\small \tt #1}}}} 
\newcommand{\Eulerian}[2]{\genfrac{<}{>}{0pt}{}{#1}{#2}}

\graphicspath{{./}}

\received{2020-03-12}
\revised{2021-04-07}       
\accepted{2021-04-07}
\publicationdetails{22}{2021}{2}{4}{6196}

\begin{document}

\author{Andrei Asinowski\affiliationmark{1}
 \and Cyril Banderier\affiliationmark{2}
 \and Benjamin Hackl\affiliationmark{1}
 }
\title[Flip-sort and combinatorial aspects of pop-stack sorting]{Flip-sort and combinatorial aspects of pop-stack sorting}
\affiliation{University of Klagenfurt, Klagenfurt, Austria\\
 University of Sorbonne Paris Nord, Villetaneuse, France
 }

\keywords{permutation, stack sorting, generating function, finite automaton, generating tree, lattice path}

\maketitle

\begin{abstract} \vspace{1mm}
Flip-sort is a natural sorting procedure which raises fascinating combinatorial questions.
It finds its roots in the seminal work of Knuth on stack-based sorting algorithms
and leads to many links with permutation patterns.
We present several structural, enumerative, and algorithmic results on
permutations that need few (resp.~many) iterations of this procedure to be sorted.
In particular, we give the shape of the permutations after one iteration,
and characterize several families of permutations related to the best and worst cases of flip-sort.
En passant, we also give some links between pop-stack sorting, automata, and lattice paths,
and introduce several tactics of bijective proofs which have their own interest.

\vspace{-1mm}

\begin{center}
\SixFlowerAltPetal\SixFlowerAltPetal\SixFlowerAltPetal\ We dedicate this article to Don Knuth, for his 1111000th birthday\dots up to a permutation! \SixFlowerAltPetal\SixFlowerAltPetal\SixFlowerAltPetal 
\end{center}
\end{abstract}



\renewcommand{\baselinestretch}{0.99}\normalsize
\renewcommand{\contentsname}{}
\tableofcontents
\renewcommand{\baselinestretch}{1.0}\normalsize
\raggedbottom
\clearpage

\section{Introduction and definitions}
\subsection{Flip-sort and pop-stack sorting}

Sorting algorithms are addressing one of the most fundamental tasks in computer science, 
and, accordingly, have been intensively studied. 
This is well illustrated by the third volume of Knuth's {\em Art of Computer Programming}~\cite{Knuth73}, which tackles many questions related 
to the worst/best/average case behaviour of sorting algorithms, 
and gives many examples of links between these questions and the combinatorial structures hidden behind permutations
(see also~\cite[Chapter 8.2]{Bona12} and~\cite{Sedgewick97}). 
For example, in~\cite[Sec.~2.2.1]{Knuth68}, 
Knuth considers the model 
of permutations sortable by one stack
(neatly illustrated by a sequence of $n$ 
wagons on a railway line, with a side track that can be used for permuting wagons)
and show that they are characterized as 231-avoiding permutations:
this result is a cornerstone of the field of permutation patterns.
Since that time, many results were also obtained e.g.~on permutations sortable with 2 stacks, 
which offer many nice links with the world of planar maps~\cite{west93, Dulucq98}.
The analysis of sorting with 3 stacks remains an open problem 
(see~\cite{Defant20} for some recent investiga\-tions). 
Numerous variants were considered (stacks in series, in parallel, etc.). 
Our article pursues this tradition
by dealing with the combinatorial aspects of the flip-sort algorithm,
a procedure for sorting permutations via a 
\textbf{pop-stack} (as considered in~\cite{AvisNewborn81,AtkinsonSack99}), which we now detail.

We use the following notation.
Permutations will be written in one-line notation,
$\pi = a_1 a_2 \ldots a_n$.
An \textit{ascent} (resp.~\textit{descent})
is a gap between two neighbouring positions,
$i$ and $i+1$, such that $a_i< a_{i+1}$
(resp.~$a_i> a_{i+1}$).
Descents of $\pi$ split it into \textit{runs}
(maximal ascending strings),
and ascents split it into \textit{falls}
(maximal descending strings).
For example, the permutation $3276145$ splits into runs
as $3|27|6|145$, and into falls as $32|761|4|5$.

Flip-sort consists of iterating flips on the input permutation,
where \textit{flip} is the transformation~$T$ that reverses \textit{all} the falls of a permutation.
For example, $T(3276145)=2316745$.
If one applies~$T$ repeatedly (thus obtaining $T(\pi), T^2(\pi), \ldots$),
then one eventually obtains the identity permutation $\id$.
For example: $3276145 \mapsto 2316745 \mapsto 2136475 \mapsto 1234657 \mapsto 1234567$;
see Figure~\ref{fig:ex00} for a visualization of this example by means of permutation diagrams.
Nota bene: if one does not impose the stack to contain increasing values, 
then the process can be related to the famous pancake sorting problem (see~\cite{GatesPapadimitriou79}).

\begin{figure}[h]
\centering\setlength{\belowcaptionskip}{-1mm}
\includegraphics[scale=1]{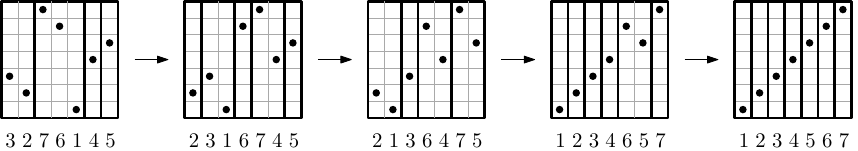} 
\internallinenumbers
\caption{The flip-sort algorithm iteratively flips all the falls until the permutation gets sorted.
In these drawings, the solid vertical bars separate the falls. Each iteration can also be seen as 
sending the input into a pop-stack: while reading the permutation from left to right, each decreasing value is added to the stack,
and \textit{the whole content} of the stack
is popped when an increasing value is met.}\label{fig:ex00}
\end{figure}
\vspace{-2mm}

\pagebreak

Let the \textit{cost} of $\pi$ be the (minimum) number of iterations needed to sort $\pi$,
that is, $\cost(\pi) = \min\{m\geq 0\colon \ T^m(\pi) = \id\}$.
To see that the cost is always finite,
observe that at each iteration the number of inversions
strictly decreases until $\id$ is reached.
It follows that, for any permutation of size $n$, we have $\cost(\pi) \leq \binom{n}{2}$. 
In fact, the cost is always lower than this naive upper bound: 
Ungar~\cite{Ungar82} proved\footnote{This result 
was conjectured by Goodman and Pollack~\cite{GoodmanPollack81}, and proved by Ungar,
in the context of \textit{allowable sequences of permutations} in links with the minimum number
of directions determined by $n$ non-collinear points in the plane; see the crisp presentation 
of this nice geometrical question in~\cite[Chapter 12]{AignerZiegler18}.} that one always has $\cost(\pi) \leq n-1$
(we refine this result in our Section~\ref{sec:width}).
Moreover, this bound is tight since 
for each $n$ there are permutations of size~$n$ with cost $n-1$
(we give some examples in Section~\ref{sec:skew}).

The pop-stack sorting leads to many interesting combinatorial questions.
Some of them can be visualized by means of the \textit{pop-stack-sorting tree}
whose nodes are all permutations of size $n$,
with directed arcs from $\pi$ to $T(\pi)$.
It is convenient to draw such a tree with levels corresponding to the distance to $\id$:
see Figure~\ref{fig:tree} that shows the tree for $n=5$
(the arcs are understood to be directed from left to right).
When, more generally, one considers this tree for any value of $n$, 
it is natural to ask for a characterization of the permutations that need few
iterations of $T$ to be sorted (i.e.~the right side of the tree), 
or many iterations (i.e.~the left side of the tree);
or the image of $T$ (i.e.~the non-leaves of the tree).
We address these questions in this article, 
thus extending previous works~\cite{Ungar82,PudwellSmith19,ClaessonGudhmundsson17}.

Another challenging problem is describing a typical shape of $T^m(\pi)$, 
as $m$ runs from $1$ to $n-1$, for random $\pi$. Computer simulations
show a very distinctive shape of the corresponding permutation diagrams:
they possess some empty areas and a cloud of points accumulating along two curves 
that tend to unite 
and ultimately converge to the identity permutation; see Figure~\ref{tab:ev1200}.
We prove a more rigorous formalization of some of these claims in this article.

\subsection{Summary of our results}\label{sec:summary}
The article is organized as follows.

In Section~\ref{sec:ImT}, we study the permutations that belong to $\im$, 
the \textbf{image} of the flip transformation. 
We find a structural characterization of such permutations,
and we prove that the generating function of these permutations is rational
when the number of runs is fixed. We then use a generating tree approach to design
some efficient algorithms enumerating such permutations. 
We also show that their generating function satisfies a curious functional equation,
and we give an asymptotic bound.

Section~\ref{sec:2ps} is dedicated to permutations
with \textbf{low cost}.
A permutation $\pi$ is \textit{$k$-pop-stack-sortable}
if $\cost(\pi) \leq k$, or, equivalently, $T^{k}(\pi)=\id$.
Avis and Newborn showed that $1$-pop-stack-sortable are
precisely the layered permutations~\cite{AvisNewborn81},
while 
$k$-pop-stack-sortable permutations 
are recognizable by an automaton (see
Claesson and Gu\dh mundsson~\cite{ClaessonGudhmundsson17}).
For $k=2$, Pudwell and Smith~\cite{PudwellSmith19} listed some conjectures
that we prove via a bijection between $2$-pop-stack-sortable permutations
to a family of lattice paths.

In contrast, in Section~\ref{sec:long}, we deal with permutations
with \textbf{high cost}, and with $\imtm$ for arbitrary $m$.
Our main result is a (tight) bound on \textbf{bandwidth} of $\tau \in \imtm$,
which provides a partial explanation of the phenomena that we observe 
in Figure~\ref{tab:ev1200}. Additionally, we find a full characterization of
$\imt{n-2}$, and some conditions for $\cost(\tau)=n-1$. We conclude with a conjecture
concerning the cost of skew-layered permutations.

\clearpage 

\begin{figure}
\centering
\includegraphics[scale=.7]{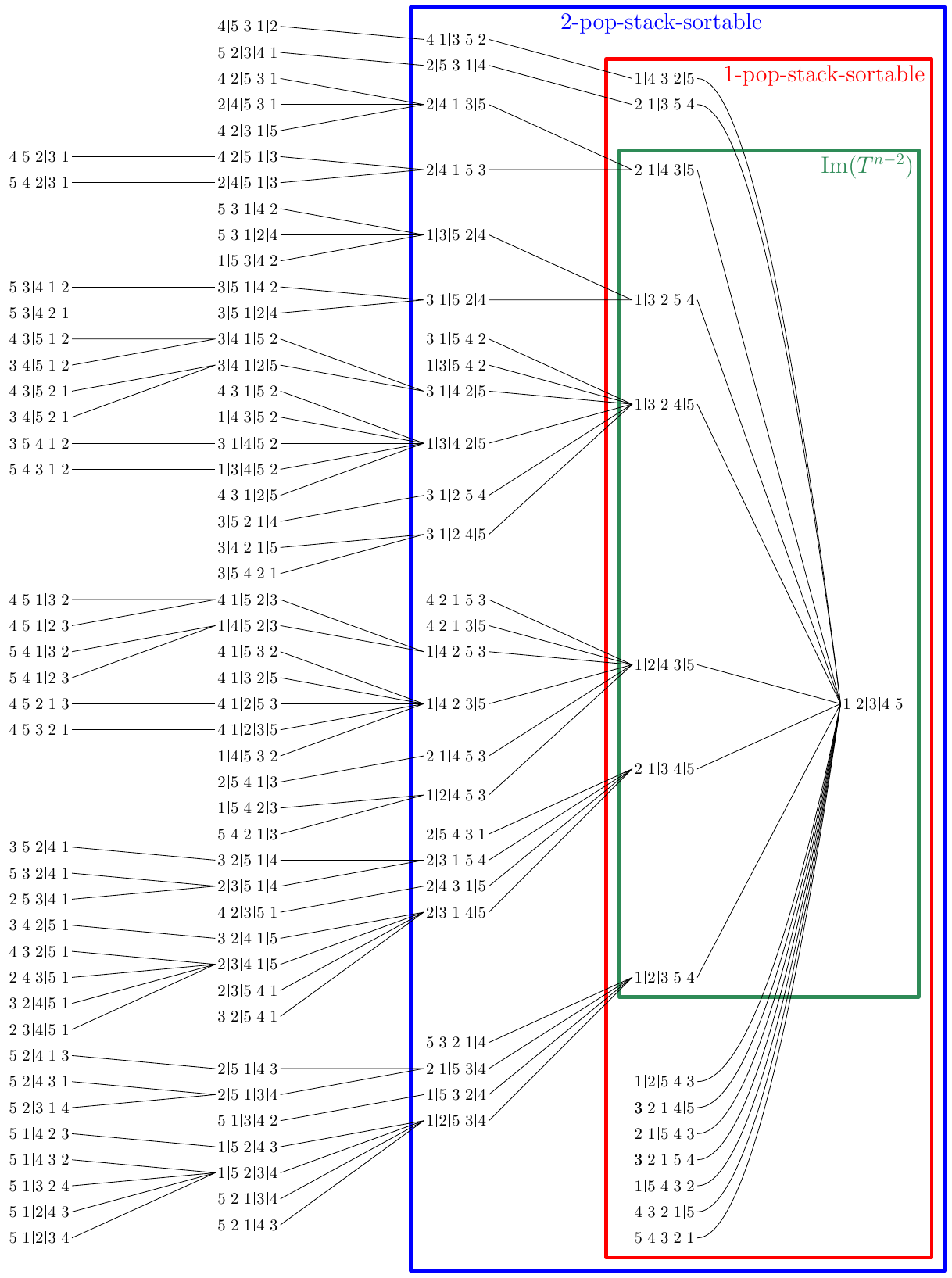}      
\internallinenumbers
\caption{The pop-stack-sorting tree for $n=5$. Flipping the falls in each permutation
leads to its successor until one reaches the identity permutation.
The leftmost column thus corresponds to the worst cases of this procedure.
In this figure, we also mark three sets considered in this article:
$1$- and $2$-pop-stack-sortable permutations, and $\imt{n-2}$. Additionally, $\imt{}$ consists of all internal nodes.}
\label{fig:tree}
\end{figure}

\clearpage

\begin{figure}
\begin{tabular}{@{\hspace{-2mm}}c@{\hspace{1mm}}c@{\hspace{1mm}}c@{}}
\includegraphics[scale=0.46]{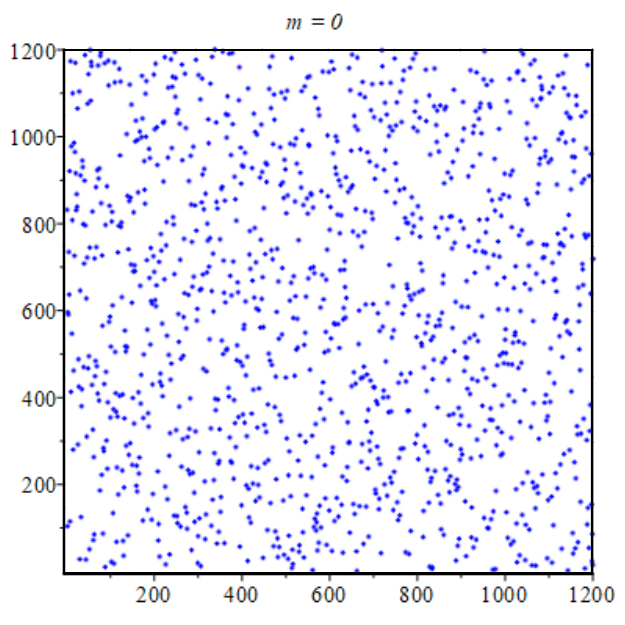} &
\includegraphics[scale=0.46]{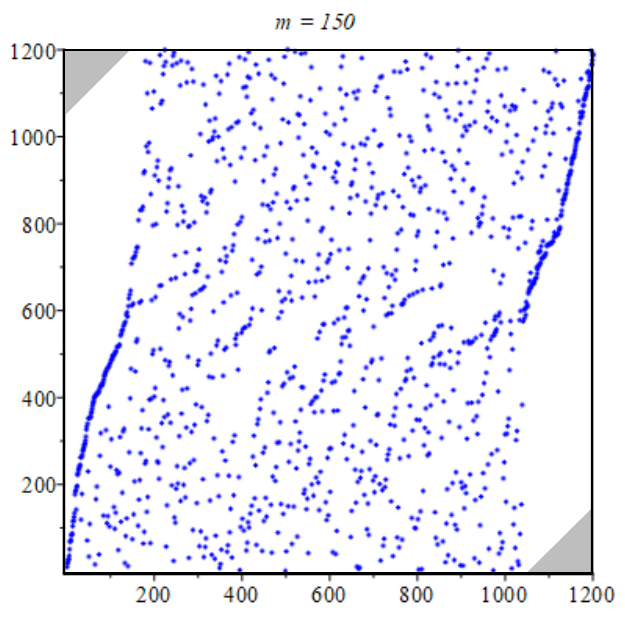} &
\includegraphics[scale=0.46]{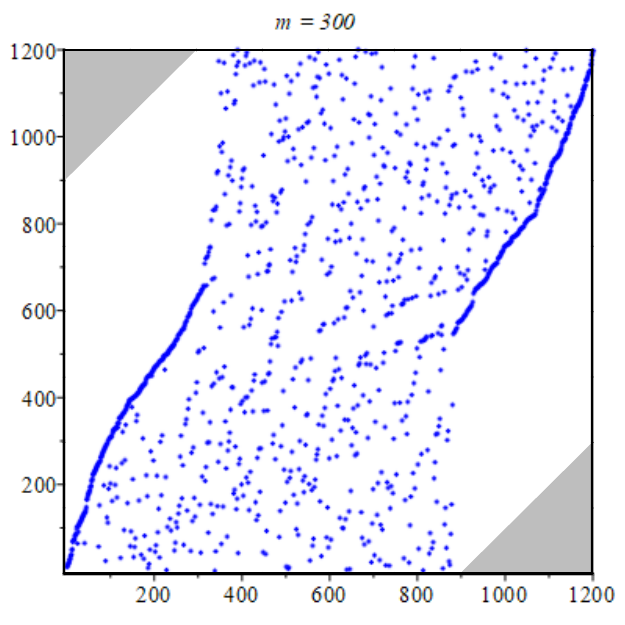} \\ 
\includegraphics[scale=0.46]{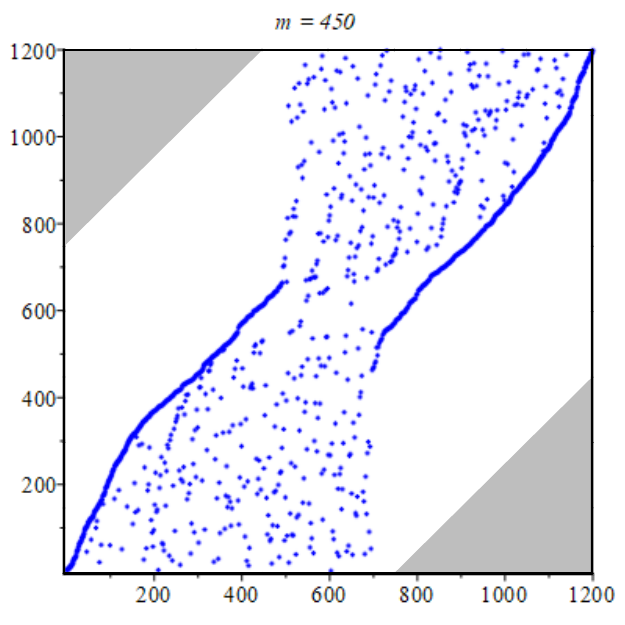} &
\includegraphics[scale=0.46]{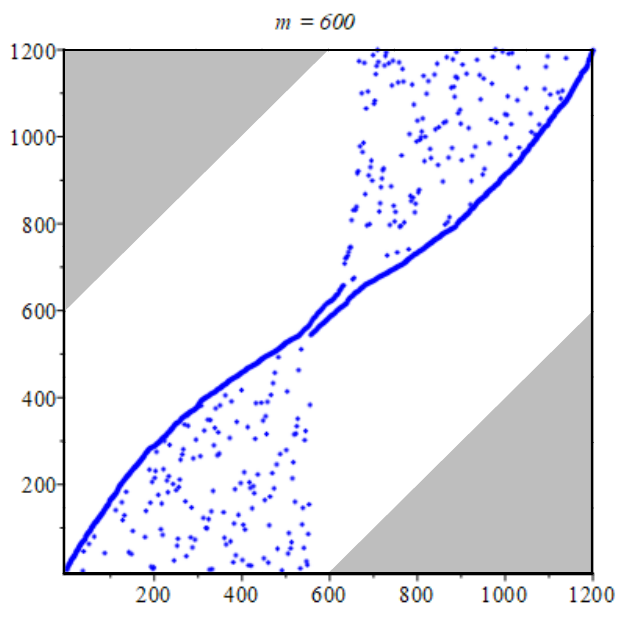} &
\includegraphics[scale=0.46]{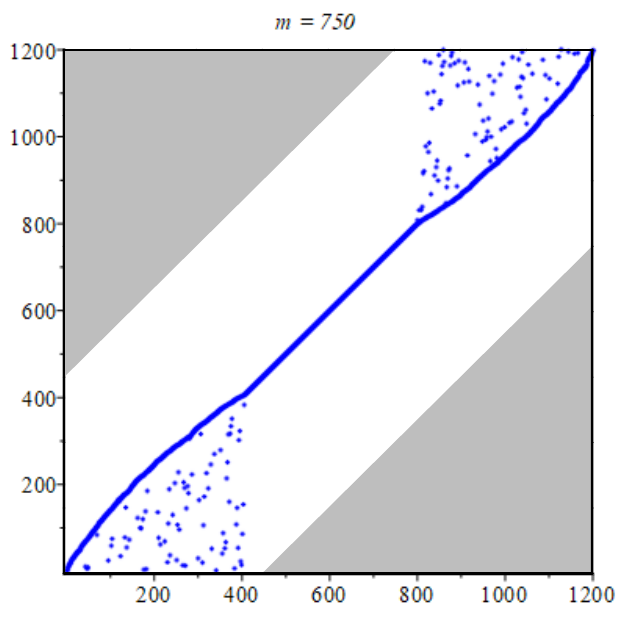} \\ 
\includegraphics[scale=0.46]{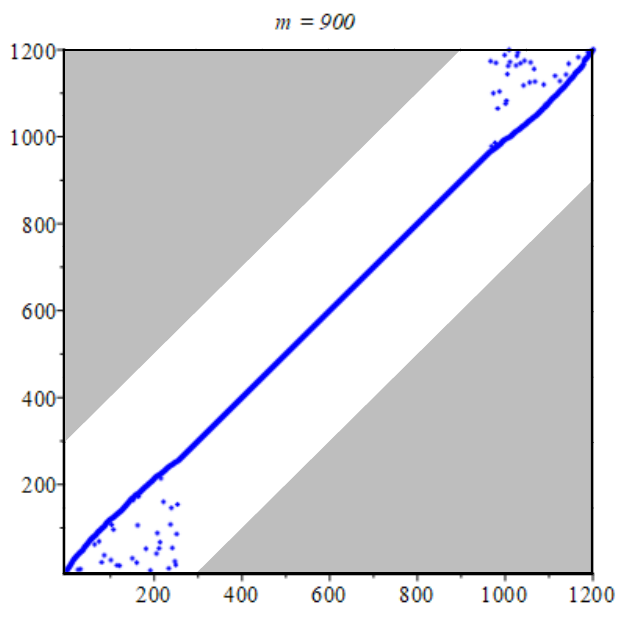} &
\includegraphics[scale=0.46]{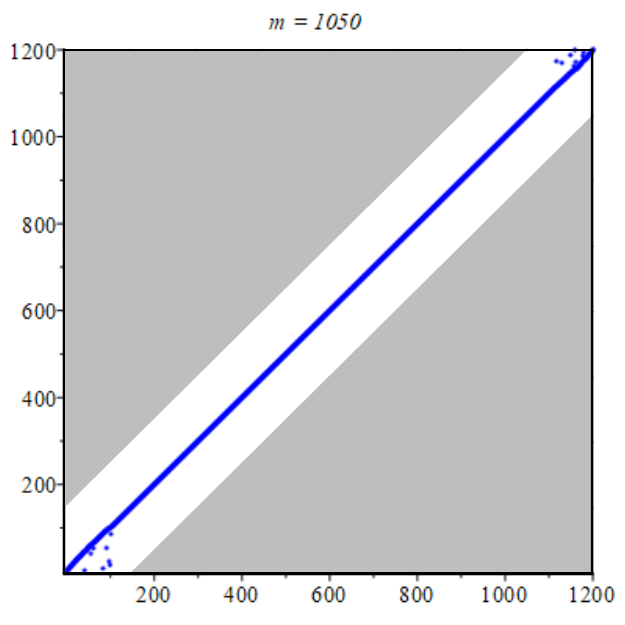} &
\includegraphics[scale=0.46]{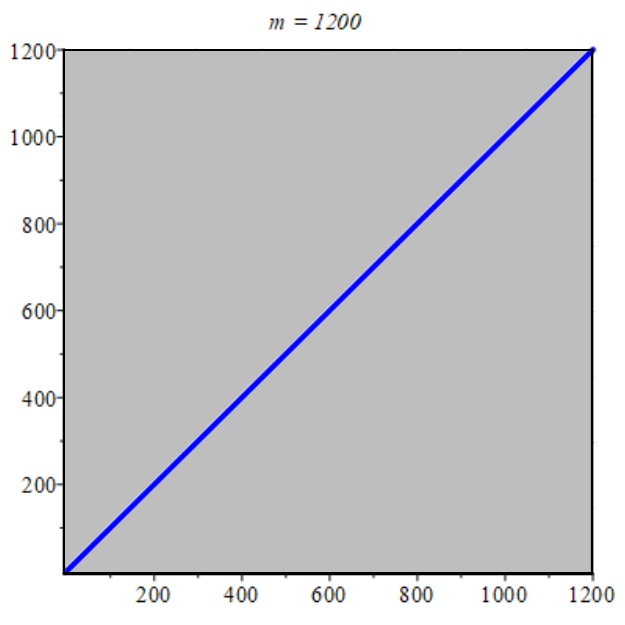} \\ 
\end{tabular}
\internallinenumbers
\caption{The evolution of the flip-sort algorithm on a sample permutation of size $n=1200$.
In this article, we show that for any input permutation
there are two areas (the grey areas in the above plots, which are proportional to the number $m$ of iterations) without any dots inside.
This entails that the permutations ``shrink'' (in the white area)
until they finally get fully sorted (after $\leq n-1$ iterations).
We also show that, for all $m$, for any coordinates in the white area, 
there exists an input permutation which will have a dot at these coordinates.
(See \href{https://lipn.fr/~cb/Papers/popstack.html}{lipn.fr/\string~cb/Papers/popstack.html} for some animations of this flip-sort process.)}
\label{tab:ev1200}
\end{figure}
\clearpage

\section{Results concerning one iteration of the flip-sort algorithm}\label{sec:ImT}

\subsection{Structural characterization of pop-stacked permutations}

We start this section with a characterization of the image of $T$, i.e.~the internal nodes in the pop-stack tree from Figure~\ref{fig:tree}.
\begin{definition}
A permutation $\tau$ is \textit{pop-stacked} if it belongs to $\imt{}$, that is,
if there is a permutation $\pi$ such that we have $\tau=T(\pi)$.
\end{definition}

\begin{definition}
A pair $(R_i$, $R_{i+1})$ of adjacent runs of~$\pi$ is overlapping if
$\mathrm{min}(R_i) < \mathrm{max}(R_{i+1})$. 
(Notice that $\mathrm{max}(R_i) > \mathrm{min}(R_{i+1})$
holds automatically since $R_i$ and $R_{i+1}$
are \textit{distinct} runs.)
\end{definition}

We begin our investigations of the image of $T$ by giving a
characterization of the permutations in $\imt{}$ in terms of
overlapping runs. In fact, in the following theorem we prove that a
permutation is pop-stacked if and only if all pairs of adjacent runs
are overlapping.
See Figure~\ref{fig:imf_ex} for a schematic drawing that represents the structure 
of permutations with overlapping adjacent runs, and an example.
Despite its very natural definition, this family of permutations
was, to the best of our knowledge, never studied before our initial conference
contributions~\cite{AsinowskiBanderierHackl19, AsinowskiBandierierBilleyHacklLinusson19}.

\begin{figure}[h]
\setlength{\abovecaptionskip}{0.6mm}
\setlength{\belowcaptionskip}{-2mm}
\centering
\includegraphics[scale=0.41]{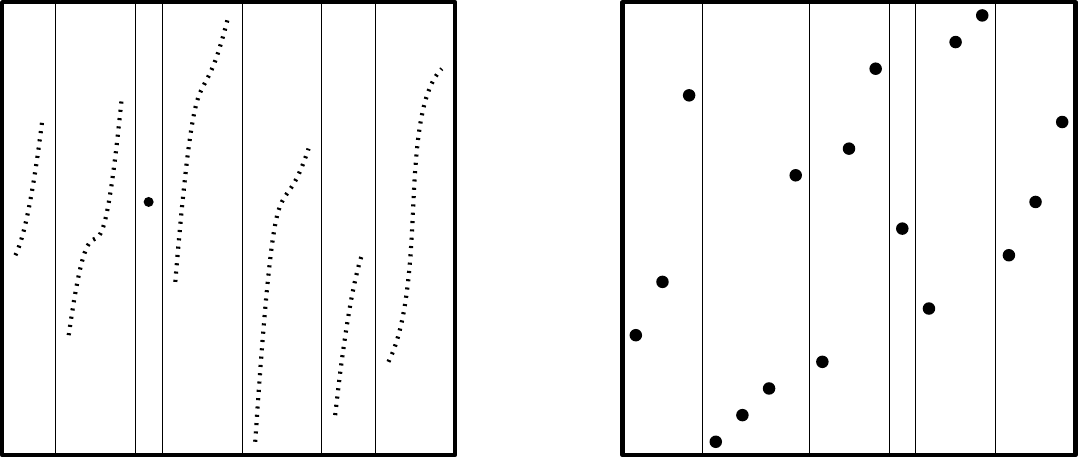}
\internallinenumbers
\caption{Pop-stacked permutations: a general schematic drawing of permutation with adjacent overlapping runs (a run can be of length $1$), 
and an example: the permutation $5 \ 7 \ 14 \, | \, 1 \ 2 \ 3 \ 11 \, | \, 4 \ 12 \ 15 \, | \, 9 \, | \, 6 \ 16 \ 17 \, | \, 8 \ 10 \ 13$.}
\label{fig:imf_ex}
\end{figure}

\begin{theorem}\label{thm:img_ch}
A permutation $\tau$ is pop-stacked if and only if 
all its pairs $(R_i$, $R_{i+1})$ of adjacent runs are overlapping.
\end{theorem}

\begin{proof} \quad 
[\textit{First part of the proof:} pop-stacked $\Rightarrow$ overlapping runs.]
Let $\tau = b_1 b_2 \ldots b_n$ be a permutation with
$\mathrm{min}(R_i) > \mathrm{max}(R_{i+1})$ for some pair
$(R_i$, $R_{i+1})$ of its adjacent runs.
Let $b_{\alpha}$ be the last letter in $R_i$
and $b_{\alpha+1}$ be the first letter in $R_{i+1}$
(that is, $R_i = [\ldots, b_{\alpha}]$,
$R_{i+1} = [b_{\alpha+1}, \ldots]$). 

Assume for contradiction that we have $\tau=T(\pi)$ for some permutation $\pi=a_1 a_2 \ldots a_n$.
In~$\pi$, we have $a_{\alpha} < a_{\alpha+1}$ because otherwise
the string $[a_{\alpha} a_{\alpha+1}]$ is a part of a fall in~$\pi$,
and upon applying $T$ we have $b_{\alpha} < b_{\alpha+1}$, which is
impossible, because $b_{\alpha}$ and $b_{\alpha+1}$ lie in different runs of~$\tau$.
Therefore, if we consider the partition of $\pi$ into falls,
then $a_{\alpha}$ is the last letter of some fall~$F_j$,
and $a_{\alpha+1}$ is the first letter of the next fall $F_{j+1}$.
However, the values of $F_j$ are a subset of those of $R_i$,
and the values of $F_{j+1}$ are a subset of those of $R_{i+1}$.
Therefore we have $a_{\alpha} > a_{\alpha+1}$, which is a contradiction
to $a_{\alpha} < a_{\alpha+1}$ observed above.

\smallskip
[\textit{Second part of the proof:} overlapping runs $\Rightarrow$ pop-stacked.]
Consider a permutation $\tau$ with
$\mathrm{min}(R_i) < \mathrm{max}(R_{i+1})$ for all pairs
$(R_i$, $R_{i+1})$ of adjacent runs.
Let $\pi$ be the permutation obtained from $\tau$ by reversal of all its runs.
Then
the partition of $\pi$ into falls is the same as
the partition of $\tau$ into runs, and, therefore, $\pi$ is
a (not necessarily unique) pre-image of $\tau$. \qedhere
\end{proof}

\subsection{Pop-stacked permutations with a fixed number of runs}\label{sec:fixed-runs}
Sections~\ref{sec:fixed-runs}--\ref{sec:asy} are dedicated to enumerative aspects of pop-stacked permutations.
Let $p_n$ denote the number of pop-stacked permutations of size $n$.
The sequence $(p_n)_{n \geq 1}$ 
starts with 
1, 
1, 
3, 
11, 
49, 
263, 
1653, 
11877,
95991,
862047,
8516221,
91782159,
1071601285,
13473914281,
181517350571,
2608383775171, 
39824825088809, 
643813226048935;
we have added it as \oeis{A307030} to the On-Line Encyclopedia of Integer Sequences.
While it is hard to compute more terms directly, the introduction of
additional parameters provides further insights. 
Specifically, in this section we consider the number of runs in pop-stacked permutations.
In particular, we show that for each fixed~$k$, the generating function for
pop-stacked permutations of size~$n$ with exactly $k$ runs is rational.

\smallskip

Let $p_{n,k}$ denote the number of pop-stacked permutations
of size $n$ with exactly $k$ runs.
The case $k=1$ is trivial:
for any size,
the only permutation with only one run is the identity
permutation, and it is pop-stacked as it is e.g.~the image of itself.
Thus, we have $p_{n,1}=1$ for each $n \geq 1$.
Note that,  for $k>1$,  $p_{n,k}$ is always an even number (indeed, listing the runs from the last one to the first one
gives an involution without fixed points among pop-stacked permutations), therefore $p_n$ is always an odd number.

One key ingredient of our further results is the following encoding of permutations,
which we call {\em scanline mapping}.
Let $\pi$ be a permutation with $k$ runs. Let $r_i$ be the index of the run
in which the letter $i$ lies. Consider the word
$w(\pi)=r_{\pi^{-1}(1)} \, r_{\pi^{-1}(2)} \, \ldots \, r_{\pi^{-1}(n)}\in \{1,\dots,k\}^n$.
Visually, we scan the graph of $\pi$ from the bottom to the top, and
for each point that we encounter in this order, we
record to which run it belongs; see Figure~\ref{fig:perm_seq}.

\bgroup
\begin{figure}[h]
\setlength{\belowcaptionskip}{-.45mm}
\centering
\includegraphics[scale=0.8]{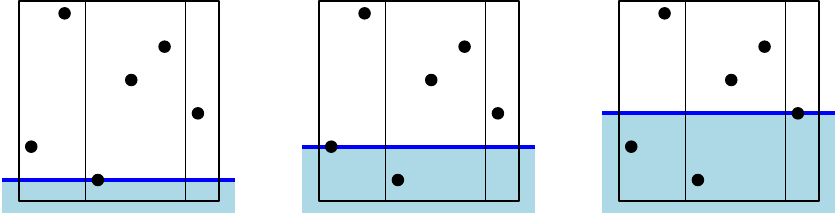}
\internallinenumbers
\caption{The {\em scanline mapping} allows us to encode permutations by words:
here, $\pi=261453$ is encoded by the word $w(\pi)=213221$ (the sequence of runs encountered if one reads the permutation bottom-to-top).}
\label{fig:perm_seq}
\end{figure}
\vspace{-2mm}
\egroup

\begin{proposition}\label{thm:seq} The scanline mapping has the following three properties.
\begin{enumerate} 
 \setlength{\itemsep}{2pt}
 \setlength{\parskip}{2pt}
\item The scanline mapping $\pi \mapsto w(\pi)$ is injective.
\item The scanline mapping $\pi \mapsto w(\pi)$ induces a bijection
between permutations of size $n$ with $k$ runs
\textnormal{and} the words in $\{1,\dots,k\}^n$ in which
there is an occurrence of $j+1$ before\footnote{N.B.: There should be some occurrence of $j+1$ \textit{somewhere} before some occurrence of $j$,
not necessarily \textit{just} before. For example, the sequence $42312$ with $k=4$
satisfies the condition. Similarly for the last item of this proposition.} an occurrence of $j$, for each $j$ ($1 \leq j \leq n-1$).
\item The scanline mapping $\pi \mapsto w(\pi)$ induces a bijection
between pop-stacked permutations of size $n$ with $k$ runs
\textnormal{and} the words in $\{1,\dots,k\}^n$ in which
there is an occurrence of $j+1$ before an occurrence of $j$,
and also an occurrence of $j$ before an occurrence of $j+1$, for each $j$ ($1 \leq j \leq n-1$).\qedhere 
\end{enumerate}
\end{proposition}\pagebreak
\begin{proof}
\begin{enumerate} 
\item The positions of $i$ ($1\leq i \leq k$) in $w(\pi)$
are the values in the $i$th run of $\pi$.
Thus, $\pi$ is reconstructed from $w(\pi)$ uniquely.
\item If (and only if) for some $j$, all the occurrences of $j$
in $w(\pi)$
are before the occurrences of $j+1$, then the corresponding positions of
$\pi$ do not form two distinct runs and thus we do not get
a permutation with $k$ runs.
\item If (and only if) for some $j$, all the occurrences of $j+1$
in $w(\pi)$
are before the occurrences of $j$, then all the values in the
$j$th run of $\pi$ are larger than all the values in the
$(j+1)$st run, and thus these runs are not overlapping.
\end{enumerate}\vspace*{-1.5\baselineskip}
\end{proof}

Proposition~\ref{thm:seq} can be used to
obtain a formula for the case of two runs directly:

\begin{proposition}\label{thm:k=2}
For $n\geq 1$,
the number of pop-stacked permutations
of size $n$ with exactly two runs,
is $p_{n,2} = 2^n-2n$.
\end{proposition}\vspace{-2.5
mm}
\begin{proof} By Proposition~\ref{thm:seq}, $p_{n,2}$
is the number of words in $\{1, 2\}^n$
with an occurrence of $1$ before an occurrence of $2$,
and also an occurrence of $2$ before an occurrence of $1$.
There are $2n$ words that violate this condition
(including the ``all-1'' and the ``all-2'' words).
\end{proof}

It is pleasant to have a combinatorial explanation for the number of pop-stacked
permutations with two runs, and this game could be pursued (algorithmically)
for a fixed number of runs $k$, but 
to get a closed-form formula holding for any arbitrary number of runs is still open.
While the following theorem does not give an explicit formula for all of those
cases immediately, it proves that the counting sequence for pop-stacked
permutations with precisely $k$ runs is nice and simple from a structural
point of view.

\begin{theorem}\label{thm:rat}
Let $k \geq 1$ be fixed. 
Then $P_k(z):=\sum_{n\geq 1} p_{n,k}z^n$, the generating function for the number of
pop-stacked permutations with exactly $k$ runs,
is rational.
\end{theorem} 
\begin{proof}
It is well known that the generating function of words recognized by an automaton is rational
(see e.g.~\cite[Sec.~I.4.2]{FlajoletSedgewick09}).
We use Proposition~\ref{thm:seq} to construct a deterministic automaton
$\mathcal{A}_k$ over the alphabet $\{1,\dots,k\}$ that precisely recognizes the words $w(\pi)$
that correspond to the pop-stacked permutations.
The states of $\mathcal{A}_k$ are labelled by pairs $(L, C)$, where
\begin{itemize}
\item $L \subseteq \{1,\dots,k\}$ indicates the already visited letters, and
\item $C \subseteq \bigcup_{j=1}^{k-1} \{ (j, j+1), (j+1, j) \}$
indicates the already fulfilled conditions
``there is an occurrence of $j$ before an occurrence of $j+1$'' resp.\
``there is an occurrence of $j+1$ before an occurrence of~$j$'', such that
\item if $j, j+1 \in L$, then at least one of $(j, j+1)$ and $(j+1, j)$ belongs to~$C$.
\end{itemize}
It is then straightforward to see that $\mathcal{A}_k$ precisely recognizes those words in $\{1, 2, \ldots, k\}^n$
that correspond bijectively to the pop-stacked permutations of size $n$ with $k$ runs by Proposition~\ref{thm:seq}.
Figure~\ref{fig:3x} shows such an automaton for $k=3$.
\end{proof}

\bgroup
\begin{figure}[ht]
\centering
\includegraphics[width=1\linewidth]{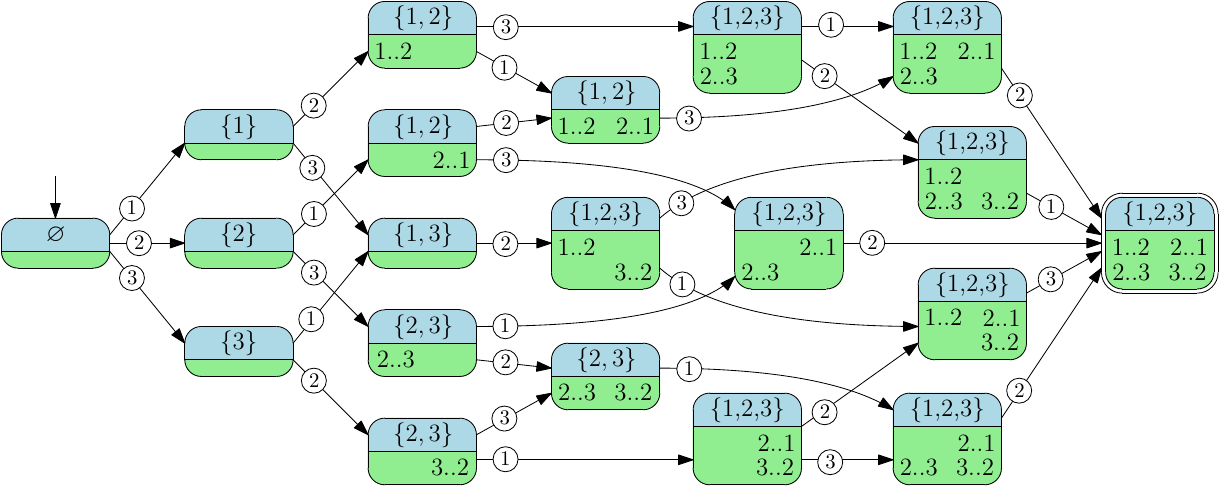}
\internallinenumbers
\caption{Automaton $\mathcal{A}_3$ that recognizes the words
encoding the pop-stacked permutations with $k=3$ runs.
In each state, the upper blue area indicates the letters from the
alphabet $\{1,\dots,k\}$ already encountered, 
and the lower green area contains $i..j$ (where $j=i\pm 1$) 
if and only if
there was already an occurrence of $i$ before an occurrence of $j$
in the word read by this point.
If a state is entered after reading the letter~$j$,
then reading $j$ again loops back to the same state:
such transitions are omitted in this~figure.}
\label{fig:3x}
\end{figure}
\egroup

\bigskip
In the next theorem, we address the complexity of $\mathcal{A}_k$:
its number of states grows roughly as $3.41^k$.
This exponential growth of the number of states also gives an insight on the complexity
of the generating functions associated to these automata.

\begin{theorem}\label{thm:states}
Denote by $a_k$ the number of states in the automaton $\mathcal{A}_k$.
The sequence $(a_k)_{k\geq 1}$ satisfies the linear recurrence $a_k = 4a_{k-1} - 2 a_{k-2}$.
Together with initial conditions, this implies that $(a_k)_{k \geq 1}$
is \oeis{A006012}: $(2, 6, 20, 68, 232, \dots)$.
Accordingly, the number of states of $\mathcal{A}_k$ grows exponentially: $a_k=\Theta( (2 + \sqrt{2})^k)$.
\end{theorem}

\begin{proof}
We proceed by induction on $k$ (starting at $k \geq 3$).
Denote by $Q_k$ the set of states of the automaton $\mathcal{A}_k$
defined in the proof of Theorem~\ref{thm:rat}.
Recall that the states are labelled by $(L, C)$, where $L$ is the set of already encountered letters,
and $C$ is the list of already fulfilled conditions of the kind $(j, j+1)$ or $(j+1, j)$.

We partition the states of $\mathcal{A}_k$ into three parts as follows.
\begin{itemize}
 \setlength\itemsep{0em}
\item[1.]
The states of $\mathcal{A}_k$ whose letters do not contain $k$: they are precisely all the states of $\mathcal{A}_{k-1}$.

\item[2.] The states of $\mathcal{A}_k$ whose letters contain $k$ but do not contain $k-1$. They
correspond bijectively to the states of $\mathcal{A}_{k-1}$ whose letters do not contain $k-1$,
and thus to all the states of $\mathcal{A}_{k-2}$:
\begin{equation}
(L, \, C) \in Q_{k-2}
\ \longleftrightarrow \
{(L, \, C) \in Q_{k-1}}
\ \longleftrightarrow \
(L\cup\{k\}, \, C) \in Q_k,
\end{equation}
where $k-1 \not\in L$.

\item[3.] Finally, the states of $\mathcal{A}_k$
whose letters contain both $k$ and $k-1$. They correspond $(3:1)$-bijectively to the states of
$\mathcal{A}_{k-1}$ whose letters contain $k-1$:
\begin{equation}
{(L, \, C) \in Q_{k-1}}
 \ \longleftrightarrow \
\left.
\begin{array}{l}
(L\cup\{k\}, \, C\cup\{(k-1, k)\}), \\
(L\cup\{k\}, \, C\cup\{(k, k-1)\}), \\
(L\cup\{k\}, \, C\cup\{(k-1, k), (k, k-1)\})
\end{array}
\right\}\in Q_k,
\end{equation}
where $k-1 \in L$.
\end{itemize}
Therefore, summing over these three cases, we have
\begin{equation}
a_k = a_{k-1} + a_{k-2} + 3(a_{k-1} - a_{k-2}),
\text{\qquad and thus } a_k=\frac{(2+\sqrt2)^k- (2-\sqrt2)^k}{\sqrt 2},\end{equation}
which completes the proof.
\end{proof}

{\bf Remark:}
Upon performing \textit{minimization} on $\mathcal{A}_{k}$,
we obtain automata with
the number of states given by $(b_k)_{k \geq 1}=(2, 6, 16, 40, 98, \dots)$.
This leads us to conjecture 
that this sequence satisfies the linear recurrence $b_k=3b_{k-1} -b_{k-2}-b_{k-3}$,
and that it is in fact \oeis{A293004}\footnote{The goddess of combinatorics is thumbing her nose at us, as 
this sequence is also related to permutation patterns in many ways: it counts permutations related to the elevator problem~\cite[5.4.8 Ex.8]{Knuth73}, permutations avoiding 2413, 3142, 2143, 3412, and some guillotine partitions~\cite{AsinowskiMansour10}!}.
Possibly, this could be proved by following the impact 
of each step of Brzozowski's algorithm for minimizing an automaton, see~\cite{Brzozowski63}.
It is interesting to notice that the exponential growth of the number of states
would then drop from $a_k = \Theta( (2 + \sqrt{2})^k) $ to
$b_k =\Theta( (1 + \sqrt{2})^k)$.

\medskip

Before we continue with our investigations towards efficient computation of the
numbers~$p_{n,k}$, we want to make some observations concerning the particular shape of the
rational generating functions $P_k(z)$. Recall that Theorem~\ref{thm:rat} provides
a construction for a deterministic finite automaton that recognizes pop-stacked
permutations of length $n$. Given such an automaton, the associated generating
function can be extracted in a straightforward way. Doing so for $1\leq k\leq 5$
yields the rational functions

\bgroup
\noindent\resizebox{.994\textwidth}{!}{
\noindent\begin{minipage}{1.011\textwidth}
\vspace{3mm}
\noindent$\begin{aligned}
P_1(z)&=\frac{z}{1-z}, \\
P_2(z)&=\frac{2z^3}{(1-z)^2(1-2z)}, \\
P_3(z)&=\frac{2z^4(1+3z-6z^2)}{(1-z)^3(1-2z)^2(1-3z)}, \\
P_4(z)&=\frac{2z^6(21-74z+5z^2+180z^3-144z^4)}{(1-z)^4(1-2z)^3(1-3z)^2(1-4z)}, \\
P_5(z)&= \frac{2z^{7} (21+ 198z - 3856z^2+ 18982z^3 - 40581z^4 + 33060z^5 + 12784z^6
- 37600z^7 +17280z^8)}{(1-z)^5(1-2z)^4(1-3z)^3(1-4z)^2(1-5z)}.
\end{aligned}$
\vspace{3mm}
\end{minipage}
}
\egroup

Some further functions can still be computed in reasonable time. However, as the
number of states in the automaton grows exponentially in $k$, this approach is not
feasible for large values~of~$k$.\linebreak 
\pagebreak

Note that the lowest degree in the numerator is $\lfloor 3k/2\rfloor$.
This is because the smallest possible permutation size for a pop-stacked permutation with $k$ runs can only be obtained by alternating
runs of lengths 1 and 2.
It is even possible to (experimentally) observe further structure in these~$P_k$,
for example they have the following partial fraction decomposition:

\begin{equation}\label{parfrac}
P_k(z)= -1 + \frac{1}{1-kz} + \sum_{j=1}^{k-1} \frac{N_{k,j}(z)}{(1-(k-j)z)^{j+1}},
\end{equation}
where $N_{k,j}$ is a polynomial of degree $j$ in $z$ and $2j$ in $k$:

\noindent\resizebox{.99\textwidth}{!}{
\noindent\begin{minipage}{1.15\textwidth}
\vspace{2mm}
\noindent$\begin{aligned}
 N_{k,1}(z)=&3 (k-1)-(k-1)(3k-4)z,\\
 N_{k,2}(z)=&-\frac{1}{2} \left( 3 k-4 \right) \left( 3k-11 \right) + \left( 9k^3-57 k^2 +102 k-44 \right) z+ \left( -\frac{9k^4}{2}+\frac {69 k^3}{2}-90 k^2+90 k-26 \right) z^2,\\
 N_{k,3}(z)=&\left(\frac{9k^3}{2}-{\frac {81 k^2}{2}}+102 k-70\right)+ \left( -\frac {27 k^4}{2}+153 k^3-\frac {1149 k^2}{2}+793 k-256\right) z+ \\ & +
\left( \frac {27 k^5}{2}-\frac {369k^4}{2}+\frac {1845 k^3}{2}-\frac {4067 k^2}{2}+1814k-406 \right) z^2\\
 &+ \left( -\frac{9k^6}{2}+72 k^5-450 k^4+1379 k^3-\frac {4223 k^2}{2}+1423k-244 \right) z^3,\\ 
 N_{k,4}(z)=& \left( -\frac {27 k^4}{8}+\frac {189k^3}{4}-\frac {1737k^2}{8}+\frac {1533k}{4}-214 \right)+ \left( \frac {27k^5}{2}-234k^4+\frac {2943k^3}{2}-4065k^2+4536k-1184 \right) z\\
&+ \left( -\frac {81k^6}{4}+\frac {837k^5}{2}-\frac {13383k^{4}}{4}+\frac {25887k^3}{2}-24395k^2+19081k-3418 \right) z^2\\
&+ \left( \frac {27k^7}{2} -324k^6+ \frac {6291k^5}{2} -15777k^4+43070k^3-61182k^2+37856k-4956 \right) z^3\\
&+ \left( -\frac {27k^8}{8}+\frac {369k^7}{4}-\frac {8433k^6}{8}+\frac {26061k^5}{4}-23493k^4+49493k^3 -57030k^2+29420k-2968\right) z^4.
\end{aligned}$
\vspace{2mm}
\end{minipage}
}

We listed the first few values of these polynomial $N_{k,j}$ in case that some clever mind could find the general pattern for any $j$.
We failed to find a generic closed-form formula, but it is noteworthy that the decomposition of Formula~\eqref{parfrac} has some similarities 
with the partial fraction decomposition of Eulerian numbers.

\begin{proposition}
Consider the Eulerian numbers $\Eulerian{n}{k}$, defined as the number of permutations
of size $n$ that have precisely $k$ runs\footnote{Sometimes, Eulerian numbers are defined
 such that $\Eulerian{n}{k}$ enumerates all permutations of size $n$ that have
 $k$ descents. Both cases (counting with respect to runs vs.\ counting
 with respect to descents) can be obtained from each other by shifting
 $k$ by one.}. 
The generating functions $E_k(z)$ for the columns of the Eulerian number
 triangle have the following partial fraction decomposition
 \begin{equation}\label{eq:Ek-partial-fraction}
 E_k(z) = \sum_{n\geq 0} \Eulerian{n}{k} z^n = \frac{1}{1 - kz} +
 \sum_{j=1}^{k-1} \frac{(-1)^j ((k-j) z)^{j-1}}{(1 - (k-j) z)^{j+1}}.
 \end{equation}
\end{proposition}
\pagebreak
\begin{proof} 
We were surprised to find no trace of this nice formula in the literature,
so we now give a short proof of it.
First, inserting $n+1$ into a permutation of size $n$ with $k$ runs either enlarges one of the runs 
if $n+1$ is inserted at its end, or it creates a new run in all other cases; this gives the classical recurrence 
\begin{equation}\label{eq:eulerian-recurrence}
 \Eulerian{n+1}{k} = (n+1-k) \Eulerian{n}{k-1} + k \Eulerian{n}{k} \text{\qquad for $n, k\geq 1$, with $\Eulerian{1}{1}=1$.}
\end{equation}
From it, it is easy to get by induction that 
\begin{equation*}
\Eulerian{n}{k} = \sum_{j=0}^{k-1} (-1)^j (k-j)^n \binom{n+1}{j}. 
\end{equation*}
We refer to~\cite{Petersen:2015:eulerian-numbers,FoataSchuetzenberger70} for possible alternative proofs and thorough surveys on Eulerian numbers.

Then, rewriting the Eulerian numbers in $E_k(z)$ by means of this formula
 and carrying out some algebraic manipulations yields
 \begin{align*}
 E_k(z) &=
 \sum_{n\geq 0} \bigg(\sum_{j=0}^{k-1} (-1)^j (k-j)^n \binom{n+1}{j}\bigg) z^n\\
 &= \frac{1}{1 - kz} + \sum_{j=1}^{k-1}\bigg((-1)^j \sum_{n\geq j-1} (k-j)^n
 \binom{n+1}{j} z^n\bigg)\\
 &= \frac{1}{1 - kz} + \sum_{j=1}^{k-1}\bigg((-1)^j ((k-j) z)^{j-1}
 \sum_{n\geq 0} \binom{n+j}{j} ((k-j) z)^n \bigg)\\
 &= \frac{1}{1 - kz} + \sum_{j=1}^{k-1}
 \frac{(-1)^j ((k-j)z)^{j-1}}{(1 - (k-j)z)^{j+1}}.
 \end{align*}
 This proves the partial fraction decomposition.\qedhere
\end{proof}

The partial fraction decomposition~\eqref{eq:Ek-partial-fraction} allows us
to make statements about the structure of~$E_k(z)$. 
We now list several of these properties, comparing them with our observations for the generating functions $P_k(z)$.
\begin{itemize}
\item The degree of the numerator of $E_k(z)$ is $\binom{k+1}{2} - 1$,
 and for $P_k(z)$ we conjecture it to be~$\binom{k+1}{2}$.
\item The leading coefficient in the numerator of $E_k(z)$ is
 $\pm\frac{1}{k} \prod_{m=1}^{k} m!$ (only the first summand in~\eqref{eq:Ek-partial-fraction}
 contributes towards the maximum degree),
 and for $P_k(z)$ we conjecture it to be $\pm\prod_{m=1}^{k} m!$.
\item For $E_k(z)$, summing the coefficients in the numerator yields
 $\pm\prod_{m=1}^{k-1} m!$ (setting $z=1$ in the numerator of $E_k(z)$ only
 collects contributions from the last summand in~\eqref{eq:Ek-partial-fraction}),
 for $P_k(z)$ we conjecture that it is $\pm 2\prod_{m=1}^{k-1} m!$.
\end{itemize}

While we think that it is unlikely that the sequence $p_{n,k}$ 
satisfies a simple linear recurrence relation as in~\eqref{eq:eulerian-recurrence},
we derive in the next section a useful, but more complex, recurrence scheme that depends on some additional
parameters.

\pagebreak

\subsection{A functional equation for pop-stacked permutations and the corresponding polynomial-time algorithm for the enumeration}
\label{sec:popstacked-polytime}
As noted above, the counting sequence $(p_n)_{n\geq 0}$ is hard to compute directly 
without introducing additional parameters. 
In~\cite{AsinowskiBanderierHackl19}, we mentioned that a generating tree approach leads 
to an efficient enumeration algorithm for pop-stacked permutations;
it relies on an additional parameter which is either the number of runs, or the final value of the permutation.
Such generating tree approaches lead to polynomial time algorithms for computing $p_n$
(see~\cite{BanderierBousquet-MelouDeniseFlajoletGardyGouyou-Beauchamps02} for further examples 
of enumerations via generating tree approaches).

When the additional parameter is the number of runs,
this gives a recurrence which encodes the addition of a new run of length $k$ to the end 
of a given permutation of size $n$ (and relabels this concatenation to get a permutation of size $n+k$).
The cost of this approach is analysed in~\cite{ClaessonGudhmundssonPantone19},
and was implemented with care on a computer cluster, allowing the computation 
of the number of pop-stacked permutations of size $n$, for all $n\leq 1000$.

When the additional parameter is the final value of the permutation,
this gives an approach that we now detail in this section. 
Based on a generation strategy where either one or two elements are added to a given permutation, we
consider the corresponding generating tree.
With this strategy, we can derive an appropriate
recurrence which also allows us to compute the sequence in polynomial time, and
 also offers a functional equation for the corresponding multivariate
generating function.

For this generation strategy, we need to keep track of a few
additional parameters in pop-stacked permutations. Let $n$, $k$, $a$, $b$, $c$
be non-negative integers such that $k\leq n$ and $1\leq a\leq b\leq c\leq n$.
Let the set $\P_{n, k; a, b, c}$ denote all pop-stacked permutations where
\begin{itemize} 
 \setlength\itemsep{-1mm}
\item $n$ denotes the length of the permutation,
\item $k$ denotes the number of runs,
\item $a$ and $c$ denote the smallest and largest element of the last run,
 respectively,
\item if $a < c$, then $b$ denotes the second-largest element of the last run; else, if $a = c$, then also $b := a = c$.
\end{itemize}
Just to give some examples, the permutation $26|1345$ is contained in
$\P_{6, 2; 1, 4, 5}$, $14|25|36$ is contained in $\P_{6, 3; 3, 3, 6}$,
and $12356|4 \in \P_{6, 2; 4, 4, 4}$.

\bigskip

The following set of rules describes how, starting from two initial
permutation sets\linebreak $\P_{1, 1; 1, 1, 1} = \{1\}$ and $\P_{2, 1; 1, 1, 2} = \{12\}$,
all other sets $\P_{n, k; a, b, c}$ can be generated. This heavily relies on
the characterization, given in Theorem~\ref{thm:img_ch}, of pop-stacked permutations.
\bgroup
\setlist[enumerate]{leftmargin=15mm}
\begin{enumerate}[(Rule 1)] 
\setlength\itemsep{0mm}
\item\label{itm:expansion:new1}
 A new run consisting of a single element is added to the end of all
 generated permutations. For all integers $i$ with $a+1\leq i\leq c$, this
 operation corresponds to an injection
 $\P_{n, k; a, b, c}\hookrightarrow \P_{n+1, k+1; i, i, i}$.
\item\label{itm:expansion:new2}
 A new run consisting of two elements is added to the end of all
 generated permutations. For all integers $i < j$ with $1\leq i\leq c$ and
 $a+2 \leq j\leq n+2$, this corresponds to an injection
 $\P_{n, k; a, b, c}\hookrightarrow \P_{n+2, k+1; i, i, j}$.
\item\label{itm:expansion:add}
 Insert a new second-largest element into the last run of all generated
 permutations. For all integers $i$ with $b+1\leq i\leq c$, this corresponds
 to the injection $\P_{n, k; a, b, c} \hookrightarrow \P_{n+1, k; a, i, c+1}$.
\end{enumerate}
\egroup
\begin{proposition}
The algorithm described above, starting with $\P_{1,1;1,1,1}$ and $\P_{2,1;1,1,2}$
and iteratively applying the expansion rules~\ref{itm:expansion:new1}, \ref{itm:expansion:new2}, and~\ref{itm:expansion:add} generates all pop-stacked permutations in
a unique way.
\end{proposition}
\begin{proof}
 In order to see that all permutations are generated by this strategy we
 study the corresponding inverse operation. Given a pop-stacked permutation $\sigma$ of length $n$
 that is different from $1$ or $12$,
 we consider its last run. Then we carry out the following operations, based on
 the length of the final run:
 \begin{itemize}
 \item If the last run is of length 1 or 2, then one deletes it. This reverses~\ref{itm:expansion:new1} and~\ref{itm:expansion:new2}.
 \item Otherwise, if the run is of length at least 3, the second-largest
 (and therefore penultimate) element is deleted. This is the
 reversal of~\ref{itm:expansion:add}.
 \end{itemize}
 After proper relabelling, this results in a shorter permutation~$\tilde\sigma$
 that is still pop-stacked. Applying the appropriate expansion rule to~$\tilde\sigma$
 constructs a set of ``successors'' also containing $\sigma$. Observe
 that by the nature of the expansion rules, only pop-stacked permutations
 can be generated (as long as we start with a pop-stacked permutation) as
 it is made sure that the last and the penultimate run always overlap.

 This proves that for every permutation (apart from $1$ and $12$) a unique
 predecessor in the generating tree can be found. The permutations $1$ and $12$
 are special in the sense that they are the only permutations for which the
 above strategy does not yield a well-defined result. At the same time, this
 implies that when starting at any other permutation $\sigma$ and iterating
 the procedure of finding the predecessor eventually leads to either $1$ or $12$.
 This proves that when starting with these permutations (in our notation, this
 corresponds to the sets $\P_{1,1;1,1,1}$ and $\P_{2,1;1,1,2}$), all other
 pop-stacked permutations can be generated by applying~\ref{itm:expansion:new1}, \ref{itm:expansion:new2},
 and~\ref{itm:expansion:add} iteratively.
\end{proof}

Now, let us turn our focus from generating these permutations to their
enumeration. Let $p_{n, k; a, b, c} := \abs{\P_{n, k; a, b, c}}$, the number of
pop-stacked permutations in $\P_{n, k; a, b, c}$. With the notation from
Section~\ref{sec:fixed-runs}, we have
\begin{equation*}
 p_{n,k} = \sum_{a, b, c\geq 0} p_{n, k; a, b, c} \qquad \text{ and }\qquad
 p_{n} = \sum_{k\geq 0} p_{n,k}. 
 \end{equation*}

The generating tree approach can be utilized to derive a functional equation
for the associated multivariate generating function
\begin{equation}\label{eq:gf-generatingtree}
 P(z, u, v_1, v_2, v_3) =
 \sum_{\substack{n, k\geq 0\\ a, b, c\geq 0}} p_{n, k; a, b, c} z^n u^k v_1^a v_2^b v_3^c,
\end{equation}
as well as a recurrence scheme the numbers $p_{n, k; a, b, c}$.

We begin with the functional equation for $P(z, u, v_1, v_2, v_3)$, which
is proved by translating the expansion rules from above to the
level of algebraic operations on generating functions.
\pagebreak
\begin{theorem}\label{thm:FE}
 The multivariate generating function $P(z, u, v_1, v_2, v_3)$ as given
 in~\eqref{eq:gf-generatingtree} satisfies the functional equation
 \begin{align}
 \begin{split}\label{eq:functional-generatingtree}
 P(z, u, v_{1}, v_{2}, v_{3})
 &= 1 + z u v_{1} v_{2} v_{3} + z^{2} u v_{1} v_{2} v_{3}^{2}\\
 & \quad + \frac{z u v_{1} v_{2} v_{3}}{1 - v_{1} v_{2} v_{3}}
 \bigg(1 - \frac{z v_{1} v_{2} v_{3}}{1 - v_{1} v_{2}}\bigg)
 \big(P(z, u, v_{1}v_{2}v_{3}, 1, 1) - P(z, u, 1, 1,
 v_{1}v_{2}v_{3})\big)\\
 & \quad + \frac{z^{2} u v_{1} v_{2} v_{3}^{2}}{(1 - v_{1} v_{2})(1 - v_{3})}
 \big(P(z, u, v_{3}, 1, 1) - P(z, u, 1, 1, v_{1} v_{2} v_{3})\big)\\
 & \quad + \frac{z^{2} u v_{1} v_{2} v_{3}^{3}}{(1 - v_{1} v_{2})(1 - v_{3})}
 \big(P(z v_{3}, u, 1, 1, v_{1} v_{2}) - P(z v_{3}, u, 1, 1, 1)\big)\\
 & \quad + \frac{z v_{2} v_{3}}{1 - v_{2}}
 \big(P(z, u, v_{1}, v_{2}, v_{3}) - P(z, u, v_{1}, 1, v_{2} v_{3})\big).
 \end{split}
 \end{align}
\end{theorem}
\begin{proof}
This functional equation is just the reflection of the fact 
that any pop-stacked permutation is either empty, 
or one of the two roots of the generating tree, namely
$1$ (encoded in~\eqref{eq:functional-generatingtree} by the monomial  $z u v_{1} v_{2} v_{3}$) 
or $12$ (encoded in~\eqref{eq:functional-generatingtree} by $z^{2} u v_{1} v_{2} v_{3}^{2}$),
or a permutation obtained by applying one the three expansions rules
to a smaller pop-stacked permutation.

Now, each of the four remaining summands in the right-hand side of~\eqref{eq:functional-generatingtree}
can be explained by considering  that  each of the expansion rules acts as a linear operator 
 on the monomials $z^n u^k v_1^a v_2^b v_3^c$ of the generating function $P$.
While the corresponding calculations are not too
 difficult, they can be a bit tedious --- which is why we choose to illustrate
 it for one of the rules, and give only the results for the remaining ones.

 Let us first consider~\ref{itm:expansion:add}, which describes the injection $\P_{n,k; a,b,c}
 \hookrightarrow \P_{n+1, k; a, i, c+1}$ for $b+1\leq i\leq c$. Starting from
 a given permutation whose associated monomial is $z^n u^k v_1^a v_2^b v_3^c$,
 this rule generates longer permutations by inserting a new second-largest
 element into the last run. On the level of monomials, this corresponds to
 the map
 \begin{equation*} z^n u^k v_1^a v_2^b v_3^c \mapsto
 \sum_{i=b+1}^{c} z^{n+1} v_1^a v_2^i v_3^{c+1} =
 z^{n+1} v_1^a v_2 \frac{v_2^b - v_2^c}{1 - v_2} v_3^{c+1}. \end{equation*}
As this operator is linear, its application to the generating function $P$ (seen as sum of monomials)
 leads to a sum which can itself also be written in terms of $P$; 
this yields a total contribution of
 \begin{equation*} \frac{z v_2 v_3}{1 - v_2} \big(P(z, u, v_{1}, v_{2}, v_{3})
 - P(z, u, v_{1}, 1, v_{2} v_{3})\big).
\end{equation*}
 Similarly,~\ref{itm:expansion:new1} yields a total contribution corresponding 
 to the summand
 \begin{equation*} \frac{z u v_{1} v_{2} v_{3}}{1 - v_{1} v_{2} v_{3}} \bigg(1 - \frac{z v_{1} v_{2} v_{3}}{1 - v_{1} v_{2}}\bigg)
 \big(P(z, u, v_{1}v_{2}v_{3}, 1, 1) - P(z, u, 1, 1, v_{1}v_{2}v_{3})\big), \end{equation*}
 and \ref{itm:expansion:new2} yields the two remaining summands on the right-hand side
 of~\eqref{eq:functional-generatingtree}.
\end{proof}

\pagebreak

This allows us to obtain a 
recurrence relation for $p_{n, k; a, b, c}$ as given in the following theorem.
Therein, for the sake of simplicity, we make use of the Iverson bracket $[\mathsf{expr}]$,
a notation popularized in~\cite{GrahamKnuthPatashnik94}, which evaluates to 1 if
$\mathsf{expr}$ is a true expression, and 0 otherwise.

\begin{theorem}[Polynomial-time computations for pop-stacked permutations]
 Let $N$, $K$, $A$, $B$, $C$ be non-negative integers.
 The number of pop-stacked permutations of length $N$ with $K$ runs where
 the last run starts with $A$, ends with $C$, and where $B$ is either the penultimate
 element (if it exists) or coincides with $A$ and $C$ as given by $p_{N, K; A, B, C}$
 satisfies the recurrence relation
 \begin{align}
 \begin{split}\label{eq:popstack-recurrence}
 p_{N, K; A, B, C}
 & = [A = B = C] \sum_{a=1}^{A-1} \sum_{b=1}^{N-1} \sum_{c=A}^{N-1}
 p_{N-1, K-1; a, b, c}\\
 & \quad + [A = B < C] \sum_{a=1}^{C-2} \sum_{b=1}^{N-2}
 \sum_{c=A}^{N-2} p_{N-2,K-1;a,b,c}\\
 &\quad + [A < B < C] \sum_{b=A}^{B-1} p_{N-1, K; A, b, C-1}
 \end{split}
 \end{align}
 together with the initial conditions $p_{1, 1; 1, 1, 1} = p_{2, 1; 1, 1, 2} = 1$.

The number of pop-stacked permutations of length $n$ with $k$ runs is thus
 \begin{equation}
 p_{n,k}=\sum_{a=1}^n \sum_{b=a}^n \sum_{c=b}^n p_{n,k;a,b,c}
\end{equation}
and the number of pop-stacked permutations of length $n$ is thus
 $p_n=\sum_{k=1}^n p_{n,k}$,
 which can be computed with $\sim n^4/8$ time-complexity and $\sim n^3/3$ memory-complexity. 
\end{theorem}

\begin{proof}
We can either go back to the combinatorial description of the
generation of permutations via the generating tree approach, or consider
the functional equation~\eqref{eq:functional-generatingtree} and extract
the coefficient of the monomial $z^N u^K v_1^A v_2^B v_3^C$ on both sides
in order to obtain the recurrence relation~\eqref{eq:popstack-recurrence} 
for $p_{n, k; a, b, c}$. 
Now, for the complexity analysis, observe that the triple sums appearing in the first and second
branch of~\eqref{eq:popstack-recurrence} can be computed more efficiently by
defining the auxiliary sequence
\begin{equation*} d_{N, K, a, A} = \sum_{b=1}^{N} \sum_{c=A}^{N} p_{N, K; a, b, c}, \end{equation*}
and note that this sequence can be computed with a backward
recurrence $A$ from $N-1$ to $1$ via 
\begin{equation}\label{eq:auxiliary-recurrence}
d_{N, K, a, A} = d_{N, K, a, A+1} + \sum_{b=a}^{A} p_{N, K; a, b, A}.
\end{equation}
Furthermore, observe that we can actually
rewrite the third branch of the recurrence~\eqref{eq:popstack-recurrence} (for $A < B-1$ and $B < C$) as
\begin{equation*} \sum_{b=A}^{B-1} p_{N-1, K; A, b, C-1} = p_{N, K; A, B-1, C} + p_{N-1, K; A, B-1, C}. \end{equation*}

\pagebreak
We are now able to rewrite the recurrence in the form

\vspace{-4mm}
\begin{align}
 \begin{split}\label{eq:popstack-recurrence:opt}
 p_{N, K; A, B, C}
 & = [A = B = C] \sum_{a=1}^{A-1} d_{N-1, K-1, a, A}\\[-3pt]
 & \quad + [A = B < C] \Big(p_{N-1, K; A, A, A} +
 \sum_{a=A}^{C-2} d_{N-2, K-1, a, A}\Big)\\[-3pt]
 &\quad + [A < B < C] \Big([A < B-1] (p_{N, K; A, B-1, C} + p_{N-1, K; A, B-1, C})\\[-3pt]
 &\qquad \qquad \qquad\qquad 
 + [A = B-1]p_{N-1, K; A, A, C-1}\Big).
 \end{split}
\end{align}

\vspace{-2mm}
This approach allows us to compute $p_{n,k}$ for $1\leq k\leq n\leq N$ in
$O(N^5)$ arithmetic operations. Furthermore, observe that the number of runs is
actually not relevant in the recurrence.
If we were only interested in $p_n$, the number of pop-stacked permutations
of length $n$, then we could drop this additional parameter: this allows us to compute $p_n$ for
$1\leq n \leq N$ in $O(N^4)$ arithmetic operations, with $O(N^3)$ simultaneous allocations in memory. 
Actually, we can be even
more precise and obtain the main asymptotic term of the number of operations:
In the computation of $p_n$ for $1\leq n\leq N$, the branches
of the recurrence~\eqref{eq:popstack-recurrence:opt} are visited for all
$1\leq A\leq B\leq C\leq N$. Investigating these branches more closely
reveals that while the case of $A = B = C$ is asymptotically negligible
for the main term, both the case of $A = B < C$ as well as $A < B < C$
contribute $\sim N^4/24$ to the total number of additions required.
Finally, adding all $p_{n; A, B, C}$ for $1\leq A\leq B\leq C$
and $1\leq n\leq N$, one gets  $\sim N^4/24$ more additions --- which means
that in total the number of additions behaves like $N^4/8$. Similar
considerations show that this approach requires $\sim N^{3}/3$ simultaneous memory allocations.
\end{proof}

A straightforward implementation of the optimized
recurrence~\eqref{eq:popstack-recurrence:opt} in SageMath~\cite{SageMath}
that stores relevant intermediate results in cache memory makes it possible
to compute the first 100 terms of the sequence $p_n$ on a standard desktop PC
in less than 3 minutes. Figure~\ref{fig:recurrence_additions} illustrates
the number of additions carried out by our strategy and confirms our assertion
regarding the main asymptotic term.

\begin{figure}[ht]\vspace{-2.5ex}
\setlength{\abovecaptionskip}{-1mm}\setlength{\belowcaptionskip}{-1mm}
 \centering
 \includegraphics[width=0.42\linewidth]{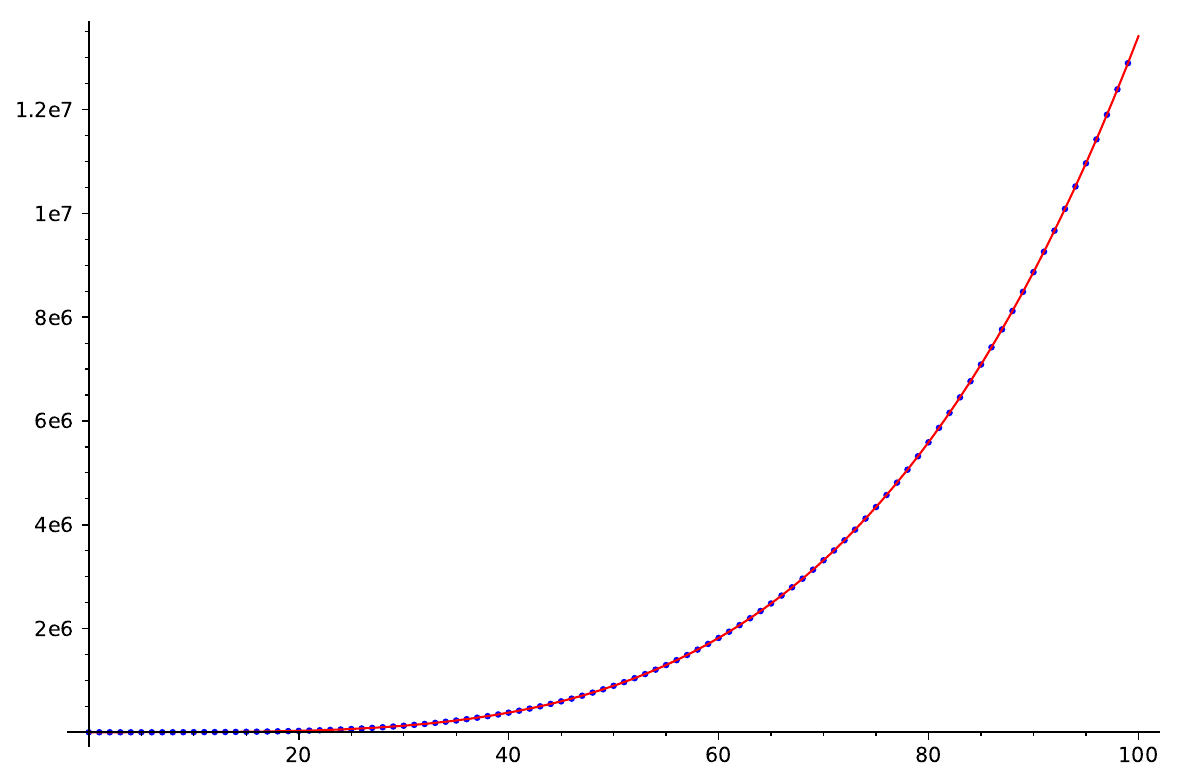}\internallinenumbers
 \caption{Number of additions required to compute $p_n$ for $1\leq n\leq 100$.
 The experimental cost (blue dots) nicely fits a quartic polynomial (in red) 
 whose main term is $n^4/8$.}
 \label{fig:recurrence_additions}
\end{figure}

However, unfortunately, neither the recurrence scheme nor the functional equation
for the generating function yield sufficient leverage to carry out an exact
analysis of the asymptotic behaviour of the sequence $p_n$. In the following
section, we provide a lower bound for the growth of $p_n$ and give pointers
to experimental observations.

\pagebreak
\subsection{Asymptotics of pop-stacked permutations}\label{sec:asy}

In many cases, the growth of restricted families of permutations is much less than~$n!$,
 and is just exponential:
for example, by the Stanley--Wilf conjecture (proved by Marcus and Tardos~\cite{MarcusTardos}), this
always holds for permutation classes defined by \textit{classical} forbidden patterns.
It is natural to ask whether this is the case for pop-stacked permutations.
We now prove that they in fact grow much faster.
\smallskip

\begin{theorem}[Superexponential growth of pop-stacked permutations]~\label{thm:asympt}
The asymptotic growth of the number of pop-stacked permutations is at least $\exp(n \ln(n) - n \ln(2)-n+o(n))$.
\end{theorem}
\begin{proof}
We achieve this by constructing an explicit class of permutations, as follows.
Assume that~$n$ is even, and consider any pair of permutations, $\pi$ and $\tau$,
each of size $n/2$.
\textit{Intertwine} them as shown in Figure~\ref{fig:asympt}, i.e.~consider the permutation $\sigma$ of size $n$ defined by
\begin{equation}
\text{for } 1 \leq j \leq n/2 : \ \ \ \sigma(2j-1)=\pi(j), \ \sigma(2j)=\tau(j).
\end{equation}

\begin{figure}[hb]\setlength{\abovecaptionskip}{0.8mm}
\setlength{\belowcaptionskip}{-1mm}
\centering
\includegraphics[scale=0.64]{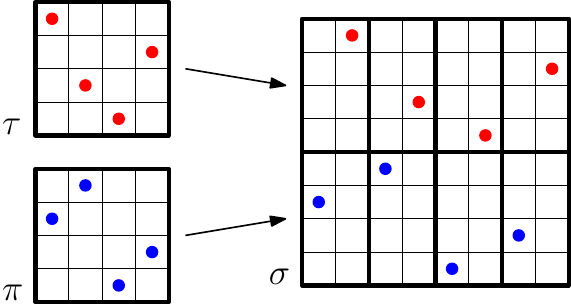}
\internallinenumbers
\caption{Intertwining two permutations gives a pop-stacked permutation and proves Theorem~\ref{thm:asympt}.}
\label{fig:asympt}
\end{figure}

It is easy to see that such $\sigma$ is necessarily pop-stacked
(it clearly satisfies the overlapping run condition from Theorem~\ref{thm:img_ch}).
Since the mapping $(\pi, \tau) \to \sigma$ is injective,
we conclude that we have at least $((n/2)!)^2$ pop-stacked permutations of size $n$.
Then, Stirling's approximation for the factorial leads to the theorem.
Note that the fact that the above argument holds for even $n$ only is not a restriction:
by inserting $n+1$ at the end of any pop-stacked permutation of length $n$, 
one obtains a new pop-stacked permutation of size $n+1$; 
 therefore the sequence $p_n$ is strictly increasing, and thus the Stirling bound given in the claim remains valid for odd $n$.
 \end{proof}

Due to the nature of the counting sequence and the fact that we have no
appropriate representation of the associated generating function, an exact
asymptotic analysis of the growth behaviour of the sequence remains a challenge.
The above theorem gives a lower bound of $.5^n n!$, it is also possible 
to compare with Andr\'e's alternating permutations to get a lower bound of $.63^n n!$ (see e.g.~\cite[p.~5]{FlajoletSedgewick09}).
Note that the authors of~\cite{ClaessonGudhmundssonPantone19} carried out an
experimental analysis using automated fitting and differential approximation.
Their analysis of the counting sequence led them to conjecture that the
corresponding exponential generating function possesses an infinite number
of singularities, thus implying non-D-finiteness. Ultimately,
they conjecture that the asymptotic growth of the sequence $p_n$ counting
pop-stacked permutations of size $n$ is 
\begin{equation*} p_n \sim (0.695688\ldots) \cdot (0.898118\ldots)^n \cdot n!; \end{equation*}
see~\cite[Section 3.3]{ClaessonGudhmundssonPantone19} for more details.

\pagebreak

\section{Best cases of flip-sort: permutations with low cost}\label{sec:2ps}
\subsection{Notation and earlier work}

Some of the earlier work on pop-stack sorting dealt with permutation with \textbf{low} cost.
The following notation was introduced there.

\begin{definition}\label{def:kPS}
A permutation $\pi$ is \textit{$k$-pop-stack-sortable} if $\cost(\pi) \leq k$
(or, equivalently, if $\imt{k}(\pi) = \id$).
\end{definition}

$1$- and $2$-pop-stack-sortable permutations
are shown by the red and the blue frames
in the pop-stack-sorting tree in Figure~\ref{fig:tree}.

Several results concerning $k$-pop-stack-sortability are known:

\begin{itemize}
\item Avis and Newborn~\cite{AvisNewborn81} 
proved that a permutation is $1$-pop-stack-sortable if and only if it is layered\footnote{A 
layered permutation is a direct sum of its falls, for example $321|654|87$. See also Definition~\ref{def:lay} below.}.
\item Pudwell and Smith~\cite{PudwellSmith19} found a structural characterization of 
$2$-pop-stack-sortable permutations, as well as a bijection between such permutations and 
\textit{polyominoes on a twisted cylinder of width~$3$} (see~\cite{AleksandrowiczAsinowskiBarequet13}). 
Moreover, they proved that the generating function for $2$-pop-stack-sortable permutations is rational
(we shall reprove their result below).
\item Claesson and Gu{\dh}mundsson~\cite{ClaessonGudhmundsson17} 
generalized the latter result showing that for each fixed~$k$, 
the generating function for $k$-pop-stack-sortable permutations is rational.
\end{itemize}

\subsection{The pre-images of \texorpdfstring{$1$}{1}-pop-stack-sortable permutations}

We begin with a nice enumerative 
result in which the notion of $1$- and $2$-pop-stack-sortable permutations
is combined with our notion of pop-stacked permutations.
Specifically, it concerns those $1$-pop-stack-sortable permutations
that belong to the image of $T$ (thus,
corresponding to internal nodes in the right column 
before the identity in the pop-stack tree from Figure~\ref{fig:tree}).

\begin{theorem}Pop-stack layered permutations have the following properties:
\begin{enumerate}
\item Let $\LI_n$ be the set of pop-stacked layered permutations
(or, equivalently, of $1$-pop-stack-sortable permutations that belong to $\imt{}$)
of size $n$. The generating function for $(|\LI_n|)_{n \geq 1}$ is
\begin{equation}\label{eq:GFLI}
 x+\frac{(x+x^2)^2}{1-(x+x^2+x^3)} = x+x^2+3x^3+5x^4+9x^5+17x^6+31x^7+\ldots. 
\end{equation}
Consequently, this family is enumerated by ``tribonacci numbers''
 (\oeis{A000213}).
\item Let $\tau \in \LI_n$.
All the permutations $\pi$ such that $T(\pi)=\tau$ can be constructed 
by the following procedure:
\begin{itemize}
\item Put \textup{primary bars} at all descents of $\tau$.
\item Optionally, put \textup{secondary bars} at some descents so that 
each primary bar has at least one neighbouring position without any bar.
\item Reverse all the blocks determined by the bars.
\end{itemize}
 
\end{enumerate}
\end{theorem}

\pagebreak
\begin{proof}\vspace{-3mm}
\begin{enumerate}\setlength\itemsep{-1mm}
\item Let $\tau$ be a layered permutation of size $n$.
Partition it into falls.
Refer to the first and the last falls as \textit{outer falls},
to all other falls as \textit{inner falls}.
For $n\geq 3$
it is easy to see that we have non-overlapping adjacent runs
if and only if
there is an outer fall of size $\geq 3$
or an inner fall of size $\geq 4$.
For $n=2$ the permutation $12$ is pop-stacked, and the permutation $21$ is not.
It follows that for $n \geq 2$ we have $\tau \in \LI_n$
if and only if it has at least two falls,
the outer falls being of size $1$ or $2$,
the inner falls of size $1$ or $2$ or $3$.
This implies the generating function~\eqref{eq:GFLI}.
\item As $\tau$ is obtained from $\pi$ by reversing all its falls,
$\pi$ is obtained from $\tau$ by reversing its runs --- possibly, further partitioned.
That is, we must first partition $\tau$ into runs (by primary bars),
then optionally further partition the runs (by secondary bars),
then reverse all the blocks.

We need to prove the condition on secondary bars.
Suppose a primary bar separates a descent $b_i|b_{i+1}$
(we have $b_i > b_{i+1}$).
If we put bars at both adjacent gaps, $|b_i|b_{i+1}|$,
then we have the same values in positions $i$ and $i+1$ in $\pi$.
However, this would be a descent, that is, a part of a fall in $\pi$, 
and it should be reversed when we apply $T$.
Otherwise, both sub-runs (that from the left and that from the right of the primary bar)
will yield, upon reversal, two distinct falls in $\pi$, 
that will recover these sub-runs when we apply $T$.
\end{enumerate} \vspace{-7mm}
\end{proof}

\textit{Example.} Let $\tau=13254687$.
The primary bars are inserted at the descents: 
$13|25|468|7$.
All the possibilities to insert secondary bars 
(for the sake of clarity, we indicate them by $\cdot$)
are $13|25|468|7$, $1{\cdot}3|25|468|7$, $13|2{\cdot}5|468|7$, $13|25|4{\cdot}68|7$, $1{\cdot}3|25|4{\cdot}68|7$,
which gives all the pre-images of $\tau$:
$31|52|864|7$, $1{\cdot}3|52|864|7$, $31|2{\cdot}5|864|7$, $31|52|4{\cdot}86|7$, $1{\cdot}3|52|4{\cdot}86|7$.
Further, it can be shown that the number of pre-images of $\tau \in \LI_n$ is always the product of some Fibonacci numbers depending on the sequence of sizes of falls of $\tau$.

\subsection{\texorpdfstring{$2$}{2}-pop-stack-sortable permutations and lattice paths}

In this section we extend some of the results by Pudwell and Smith~\cite{PudwellSmith19}.
Namely, we reprove one of their theorems and prove 
(a generalization of) one of their conjectures. 
All this is done via the uniform framework of lattice paths.
By doing this, we not only construct a bijection between $2$-pop-stack
sortable permutations and lattice paths, but also benefit from the fact 
that the theory of lattice paths is well developed (see e.g.~\cite{BanderierFlajolet02}), 
thus offering additional structural insight.

First, we reprove the theorem by Pudwell and Smith in a more combinatorial way.
Then we construct a bijection between $2$-pop-stack-sortable permutations and
a certain family of lattice paths, which enables us to prove 
two of their conjectures.

\begin{theorem}\label{thm:PS} For $n \geq 1$, $0 \leq k \leq n-1$, let $a_{n,k}$ be 
the number of $2$-pop-stack-sortable permutations of size~$n$ with exactly $k$ ascents.
\begin{enumerate}
\item \textbf{(Pudwell and Smith~\cite{PudwellSmith19}, Thm.~2.5)}
Let $A(x,y)$ be the bivariate generating function $A(x,y) =\sum a_{n,k} x^n y^k$. Then we have
\begin{equation}\label{eq:2PSbivariate}
A(x,y) = \frac{x(1+x^2y)}{1-x-xy-x^2y-2x^3y^2}.
\end{equation}
\item \textbf{(For $k=0$, Pudwell and Smith~\cite{PudwellSmith19}, Conj.~4.3)} 
Let $k \geq 0$.
The generating functions
for ``diagonal-parallel'' arrays are
\begin{equation}\label{eq:7x}
\begin{array}{l}
\displaystyle{D_k(x)=\sum_{n \geq 0} a_{2n+k+1, n+k}x^n = 
\sqrt{\frac{1+x}{1-7x}} \Bigg( \frac{1-x-\sqrt{(1+x)(1-7x)}}{2x} \Bigg)^k
},\\
\displaystyle{D_{-k}(x)=\sum_{n \geq 0} a_{2n+k+1, n}x^n = 
\sqrt{\frac{1+x}{1-7x}} \Bigg( \frac{1-x-\sqrt{(1+x)(1-7x)}}{2x(1+2x)} \Bigg)^{k}
}.
\end{array}
\end{equation}
For the ``diagonal array'' $a_{2n+1, n}$ (that is, for $k=0$) we have explicit formulas
\begin{equation}\label{eq:a_n}
a_{2n+1, n} =
\sum_{i=0}^{n-1} (-1)^i 2^{n-i} \binom{2(n-i)}{n-i} \binom{n-1}{i}
=\sum_{k\geq 0} \binom{n}{2k} \binom{2k}{k} 2^{2k+1}
 3^{n-2k-1} \bigg(2 - \frac{k}{n}\bigg).
\end{equation}
\end{enumerate}
\end{theorem}

The two-dimensional array $a_{n,k}$ is shown in Figure~\ref{fig:a}.
The $n$th row ($n \geq 1$) contains coefficients $a_{n, 0}, a_{n, 1}, \ldots, a_{n, n-1}$.
The ``diagonal-parallel'' arrays from part 2 are vertical in this drawing
($k=0$ corresponds to $1, 4, 20, 116, \ldots$, 
$k=1$ to $1, 8, 48, 296, \ldots$, etc.).

\begin{figure}[h]
\setlength{\belowcaptionskip}{-1mm}
\centering
\includegraphics[scale=.8]{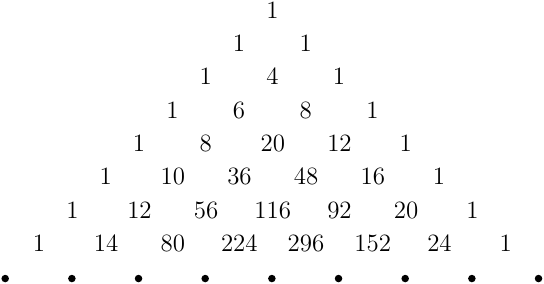} 
\internallinenumbers
\caption{The two-dimensional array $a_{n,k}$, $n \geq 1$, $0 \leq k \leq n-1$.
The coefficient $a_{n,k}$ is the number of $2$-pop-stack-sortable permutations with exactly $k$ ascents.}
\label{fig:a}
\end{figure}

\begin{proof}
\begin{enumerate}
\item 
It is shown in~\cite[Lemma~2.1]{PudwellSmith19} that 
a permutation $\pi$ is $2$-pop-stack-sortable 
if and only if for each pair of its adjacent falls, $F_i$ and $F_{i+1}$, 
one has $\max(F_i) \leq \min(F_{i+1})+1$.\footnote{Notice 
that layered permutations are similarly characterized by $\max(F_i) = \min(F_{i+1})-1$.}
Consider an ascent that separates two adjacent falls $F_i$ and $F_{i+1}$.
We say that this ascent is \textit{regular} if $\max(F_i) < \min(F_{i+1})$,
and \textit{twisted} if $\max(F_i) = \min(F_{i+1}) +1$.
An ascent can be twisted only when at least one of the falls $F_i$ or $F_{i+1}$
is of size $>1$ (equivalently, when at least one of the adjacent gaps is a descent).

As Pudwell and Smith show, $2$-pop-stack-sortable permutations 
are bijectively encoded by sequences of their ascents and descents
(we use $\xa$ for ascents, $\xd$ for descents), where for each $\xa$ 
that has an adjacent $\xd$ it is indicated whether this $\xa$ is regular or twisted.
We find it convenient to use instead lattice paths:
we replace each $\xa$ by the \textit{up-step} $\U=(1,1)$,
$\xd$ by the \textit{down-step} $\D=(1,-1)$,
and \textit{corner-adjacent} $\U$-steps 
(that is, $\U$-steps which have at least one adjacent $\D$-step) will be bicoloured: 
$\U$-steps that correspond to a regular ascent will be coloured black,
$\U$-steps that correspond to a twisted ascent will be coloured red.
In this way, we obtain a certain family of Dyck walks.
The length of a walk is smaller by $1$ than the size of the corresponding permutation,
and the final altitude is $\#(\xa)-\#(\xd)$.
See Figure~\ref{fig:2PS_LP} for illustration.

\begin{figure}[h]
\centering
\includegraphics[scale=1]{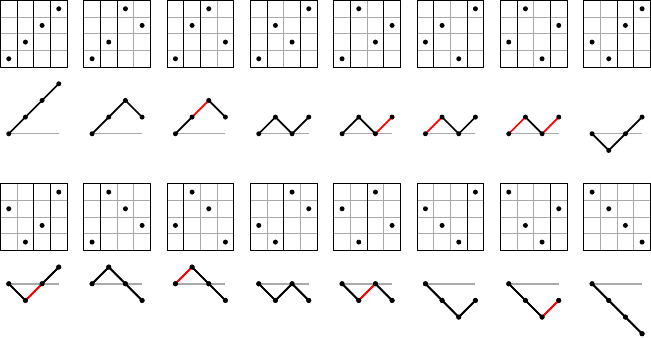} 
\internallinenumbers
\caption{Encoding of $2$-pop-stack-sortable permutations by Dyck walks
with bicoloured corner-adjacent $\U$-steps.}
\label{fig:2PS_LP}
\end{figure}

To enumerate such walks, we make use of the symbolic method
(see~\cite[Chapter I]{FlajoletSedgewick09}) to
find a combinatorial specification for (non-coloured) Dyck walks in which 
corner-adjacent $\U$-steps are marked by $y$:
\begin{equation*}
\mathcal{W}=
\epsilon \
\cup \
\mathsf{Seq}^{\geq 1} (\U) \times
\bigg(
\epsilon \,
\cup \,
 y
\mathsf{Seq}^{\geq 1} (\D) \times
\left(
\epsilon \,
\cup \,
\U
y \times \mathcal{W}
\right)
\bigg) \
\cup \
\mathsf{Seq}^{\geq 1} (\D) \times
\bigg(
\epsilon \,
\cup\,
\U
y \times \mathcal{W}
\bigg).
\end{equation*}

 This leads to the following functional equation for the
trivariate generating function $W(t,u,y)$, where 
$t$ is the variable for length,
$u$ for final altitude,
$y$ for occurrences of corner-adjacent $\U$-steps:
\begin{equation}\label{eq:LP2Weq}
W = 1 + \frac{tu}{1-tu}\Bigg( 1+ y \frac{t/u}{1-t/u} \bigg( 1 + tuy W \bigg) \Bigg) +
\frac{t/u}{1-t/u} \Bigg(1 + tuy W \Bigg).
\end{equation}
This yields the generating function
\begin{equation}\label{eq:LP2Wsol} 
W(t,u,y)=\frac{u\big( 1-t^2(1-y) \big)}{u-t-tu^2+t^2u(1-y)+t^3u^2y(1-y)},
\end{equation}
and the bivariate generating function for those walks in which corner-adjacent $\U$-steps
are \textit{bicoloured}, is
\begin{equation}\label{eq:LP2Wsol2} 
W(t,u,2)=\frac{u\big( 1+t^2 \big)}{u-t-tu^2-t^2u-2t^3u^2}.
\end{equation}
Upon performing the transformation that corresponds to the way in which we rearranged the array
of the coefficients,
we obtain $xW(x\sqrt{y}, \sqrt{y},2) = A(x,y)$. 
This completes the proof of part 1.

\item First, we find the generating function $E(t, y)$ for excursions --- 
that is, those walks that stay (weakly) above and terminate at the $t$-axis
--- in our model of Dyck walks with marked corner-adjacent $\U$-steps.
The functional equation for $E(t,y)$ is 
\begin{equation}\label{eq:LP2Eeq}
E=1+\frac{t^2\big(y+(E-1)\big)}{1-t^2yE},
\end{equation}
which yields 
\begin{equation}\label{eq:LP2Esol} 
E(t,y)=\frac{1-t^2(1-y)-\sqrt{\big(1-t^2(1-y)\big)\big(1-t^2(1+3y) \big)}}{2t^2y}.
\end{equation}
Therefore, the generating function for Dyck excursions with \textit{bicoloured} corner-adjacent $\U$-steps is
\begin{equation}\label{eq:LP2Esol2} 
E(t,2)=\frac{1+t^2-\sqrt{(1+t^2)(1-7t^2)}}{4t^2}.
\end{equation}
(In fact, we have $E(t,y) = 1+\frac{t^2y}{1-t^2} M \Big(\frac{t^2y}{1-t^2} \Big)$, 
where $M(x)$ is the generating function for Motzkin numbers. This can be shown bijectively, 
starting from the fact that the number of Dyck excursions
of semilength $n$ with exactly $k$ corner-adjacent $\U$-steps, is $\binom{n-1}{k-1}M_{k-1}$,
where $M_{k-1}$ is the $(k-1)$st Motzkin number. 
A similar result on Dyck paths with marked pattern $\U\D\U$
was proved by Sun in~\cite[Thm. 2.2]{Sun04}.)

\smallskip

The ``diagonal coefficients'' $a_{2n+1, n}=[x^{2n+1}y^{n}]$ of $A(x,y)$ 
correspond to coefficients $[t^{2n} u^0]$ of $W(t,u,2)$. Therefore we look for 
the generating function $B(t,y)$ for \textit{bridges} --- 
that is, those walks that terminate at the $t$-axis
--- in our model. It can be routinely found by residue analysis
(extracting $[u^0]$ in $W(t,u,y)$). However, we give a more structural proof.
We decompose a bridge into an alternating
sequence of excursions and anti-excursions ($=$~rotated by $180^\circ$ excursions that stay weakly
below the $x$-axis). Since we use only $\U$ and $\D$ steps, such a decomposition is unambiguous.
Moreover, the generating function for anti-excursions is the same as for excursions.
When we join excursions and anti-excursions, new corner-adjacent $\U$-steps are never created.
Therefore we have
\begin{equation}\label{eq:LP2Bsol} 
B(t,y)=\frac{E(t,y)}{1-(E(t,y)-1)}=
\frac{\sqrt{1-t^2(1-y)}}{\sqrt{1-t^2(1+3y)}},
\end{equation}
and finally
\begin{equation}\label{eq:LP2Bsol2} 
B(t,2)=\frac{\sqrt{1+t^2}}{\sqrt{1-7t^2}} = 1+4t^2+20t^4+116t^6+708t^8+4452t^{10}+\ldots.
\end{equation}

Notice that we can rewrite this as 
\begin{equation*}B(t,2)=
\frac{1}{\sqrt{1- \frac{8t^2}{1+t^2} }}
 = \frac{1}{\sqrt{1-4x}}\bigg|_{x=2t^2/(1+t^2)},\end{equation*}
and the coefficients of $1/\sqrt{1-4x}$ are well known to be central binomial coefficients.
This directly leads to the first closed-form formula of Equation~\eqref{eq:a_n} for $a_{2n+1,n}$.
The second equivalent closed-form formula is proved easily via closure properties from holonomy theory; see~\cite{PetkovsekWilfZeilberger96}.

Finally, we consider generating functions $W_k(t,2)$ for walks that terminate at fixed altitude~$k$.
For unmarked paths we have the classical decomposition
$W_k(t)=B(t) (tE(t))^k$. Taking care of corner-adjacent $\U$-steps,
we obtain, for $k \geq 0$, 
\begin{equation*}
W_{+k}(t,2) 
=
\frac{\sqrt{1+t^2}}{\sqrt{1-7t^2}}
\left(\frac{1-t^2-\sqrt{(1+t^2)(1-7t^2)}}{2t}\right)^k,
\end{equation*}
and
\begin{equation*}
W_{-k}(t,2) 
=
\frac{\sqrt{1+t^2}}{\sqrt{1-7t^2}}
\left(\frac{1-t^2-\sqrt{(1+t^2)(1-7t^2)}}{2t(1+2t^2)}\right)^k,
\end{equation*}
which, upon due modification, yields~\eqref{eq:7x}.
\end{enumerate}\vspace{-10mm}
\end{proof}

\bigskip
\bigskip

\begin{minipage}{.5\textwidth}\internallinenumbers
Finally, if we look again at \eqref{eq:LP2Wsol2} and rewrite it in the form
$W(t,u,2)= \frac{1+t^2}{1-\frac{t}{u}-tu-t^2-2t^3u}$
then it is immediate from the theory of lattice paths
that it is also the bivariate generating function for lattice paths
that start either at $(0,0)$ or at $(2,0)$,
and use steps
$(1,-1)$, $(1,1)$, $(2,0)$, and bicoloured $(3,1)$.

\begin{center}
 \includegraphics[scale=0.8]{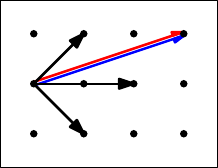}
 \end{center}

These paths can also be used to obtain formulas~\eqref{eq:7x} for diagonal-parallel 
arrays, and the computations are even easier 
because the bicolouration of some steps does not require considering their adjacent steps.
\end{minipage}

\vspace{-8.5cm} \hspace{8.5cm}
\includegraphics[scale=0.8]{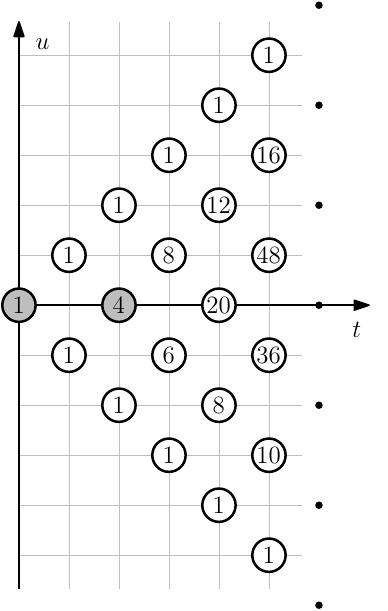}

\pagebreak
\section{Worst-cases of flip-sort: permutations with high cost}\label{sec:long}

\begin{center}
\begin{minipage}{0.76\textwidth}
``\textit{Conventional wisdom might say that extremal problems tend to be very
difficult to solve exactly if the extreme configurations are so diverse and
irregular.}''\\ Martin Aigner \& G\"unter Ziegler, in {\em Proofs from THE BOOK}.
\end{minipage}
\end{center}

This claim by Aigner and Ziegler is concerning an extremal problem in geometry 
which is in fact solved by understanding the worst cases of flip-sort!
We refer to~\cite[Chapter 12]{AignerZiegler18} and~\cite{Ungar82} for more details on this geometric problem.
In this section, we strongly extend these studies by proving structural and enumerative results 
concerning the evolution of permutations during the flip-sort process,
focussing on permutations with worst costs.
We start with some definitions, notation, and basic observations.

\subsection{Definitions and notation}\label{sec:3def}

By $\id_n$ we denote the identity permutation of size $n$,
and by $-\id_n$ we denote its reversal:
$\id_n = 12 \ldots n$, and $-\id_n = n \ldots 21$.

\begin{definition}[Layered and skew-layered permutations]\label{def:lay}
The \textit{layered permutation} $\oplus(m_1, \ldots, m_s)$
is 
the direct sum of $-\id_{m_1}, \ldots, -\id_{m_s}$.
The \textit{skew-layered permutation} $\ominus(m_1, \ldots, m_s)$
is 
the skew sum of $\id_{m_1},\ldots, \id_{m_s}$.\footnote{We refer to
\textit{Permutation patterns: basic definitions and notation}
by David Bevan~\cite{bevan2015}
for the definitions of direct sum and of skew sum.}
For example, $\oplus(2,1,3)=213654$ and $\ominus(2,1,3)=564123$.
\end{definition}

\begin{definition}[$k$-shadows]
Let $\pi$ be a permutation of size $n$,
and let $k$ be a fixed number, $1 \leq k \leq n-1$.
Let $S:=\{1, \ldots, k\}$, $L:=\{k+1, \ldots, n\}$
($S$ for \textit{small}, $L$ for \textit{large}).
The \textit{$k$-shadow} of~$\pi$
is defined to be the $\{\xs, \xl\}$-word $\sh_k(\pi)$
obtained from $\pi$ by replacing all the numbers from~$S$ by $\xs$,
and all the numbers from $L$ by $\xl$.
When $k$ is fixed, we shall usually omit the subscript and write just $\sh(\pi)$.
\textit{Example:} for $\pi=6317524$ 
we have $\sh_3(\pi)=\xl \xs \xs \xl \xl \xs \xl$.

For $1 \leq j \leq k$,
we denote by $s_j(\pi)$ the position of the $j$th occurrence of $\xs$
counted \textbf{from left} in $\sh(\pi)$;
for $1 \leq j \leq n-k$,
we denote by $\ell_j(\pi)$ the position of the $j$th occurrence of $\xl$
counted \textbf{from right} in $\sh(\pi)$.
In the example just considered, we have
$s_1(\pi)=2$,
$s_2(\pi)=3$,
$s_3(\pi)=6$;
$\ell_1(\pi)=7$,
$\ell_2(\pi)=5$,
$\ell_3(\pi)=4$,
$\ell_4(\pi)=1$.
\end{definition}

\begin{definition}[$\xs$- and $\xl$-paths, $k$-wiring diagram]
Let $\pi$ and $k$ be as above.
For $1 \leq j \leq k$, the \textit{$j$th $\xs$-path} is the sequence $(s_j(T^m(\pi)))_{m\geq 0}$.
For $1 \leq j \leq n-k$, the \textit{$j$th $\xl$-path} is the sequence $(\ell_j(T^m(\pi)))_{m\geq 0}$.
The \textit{$k$-wiring diagram} is a drawing on the rectangular grid
with infinitely many rows, labelled (from above) by $0, 1, 2, \ldots$,
that correspond to $\pi, T(\pi), T^2(\pi), \ldots$;
and $n$ columns labelled by $1, 2, \ldots, n$.
In this drawing,
we represent the $j$th $\xs$-path ($1 \leq j \leq k$)
by a polygonal \textcolor{red}{red line} connecting
the points $(m, s_j(T^m(\pi)))_{m \geq 0}$,
and
we represent the $j$th $\xl$-path ($1 \leq j \leq n-k$)
by a polygonal \textcolor{blue}{blue line} connecting
the points $(m, \ell_j(T^m(\pi)))_{m \geq 0}$.
The $j$th $\xs$-path (resp.\ $\xs$-path) of $\pi$ will be denoted by $\xs_j(\pi)$ (resp.\ by $\xl_j(\pi))$.
See Figure~\ref{fig:wd_ex} for an example.

Remark: Since we show below that after the $(n-1)$st row the paths stay in the same position
(namely, the $\xs_j(\pi)$ in the position $j$,
and the $\xl_j(\pi)$-path in the position $n+1-j$),
we draw $k$-wiring diagrams with only $n$ rows,
from the one corresponding to $\pi$ to the one corresponding to $T^{n-1}(\pi)$.
\end{definition}
\pagebreak

\begin{figure}[ht]
\centering
\includegraphics[scale=0.8]{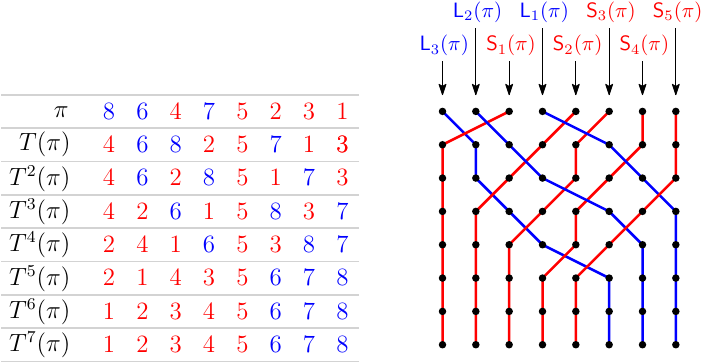}\internallinenumbers
\caption{The $k$-wiring diagram for $\pi = 86231745$ and $k=5$.}
\label{fig:wd_ex}
\end{figure}

In the following considerations, the $k$-wiring diagram
of the permutation $\rho_k:=\ominus(n-k, k)$ will play an important role.
When $k$ is fixed, we shall usually omit the subscript and write just $\rho$.
It is easy to verify that $\sh(\rho)=\xl^{n-k}\xs^k$,
and that for each $m\geq 0$ the word
$\sh(T^{m+1}(\rho))$ is obtained from $\sh(T^{m}(\rho))$
by replacing each consecutive occurrence of $\xl\xs$
by $\xs\xl$.
This yields
\begin{equation}\label{eq:wd_rho}
\begin{tabular}{c}
$s_j(T^m(\rho))=
\left\{
\begin{tabular}{lll}
$n-k+j$, & & $0 \leq m \leq j-1$, \\
$n-k+2j-m-1$, & & $j-1 \leq m \leq n-k+j-1$, \\
$j$, & & $n-k+j-1 \leq m$;
\end{tabular}
\right.$\\
\mbox{} \\
$\ell_j(T^m(\rho))=
\left\{
\begin{tabular}{lll}
$n-k+1-j$, & & $0 \leq m \leq j-1$, \\
$n-k-2j+m+2$, & & $j-1 \leq m \leq k+j-1$, \\
$n+1-j$, & & $k+j-1 \leq m$.
\end{tabular}
\right.$
\end{tabular}
\end{equation}
Accordingly,
the $k$-wiring diagram of $\rho$ has a very distinctive shape,
with $\xs$- and $\xl$-paths typically consisting of three segments.
This is illustrated in Figure~\ref{fig:wd_rho}
(since we later combine the $k$-wiring diagram of the fixed permutation $\rho$
with that of an arbitrary permutation $\pi$,
the $\xs$- and $\xl$-paths of $\rho$ will be shown by thick pink and light-blue lines).
Finally, we notice that, since the sets $S$ and $L$ are already sorted in $\rho$,
all the values along $\xs_j(\rho)$ are $j$,
and
all the values along $\xl_j(\rho)$ are $n+1-j$,
so that for $\rho=\ominus(n-k,k)$
the $k$-wiring diagram
coincides with the \textit{usual} wiring diagram where paths connect occurrences of the same value.
\begin{figure}[ht]
\centering
\includegraphics[scale=0.85]{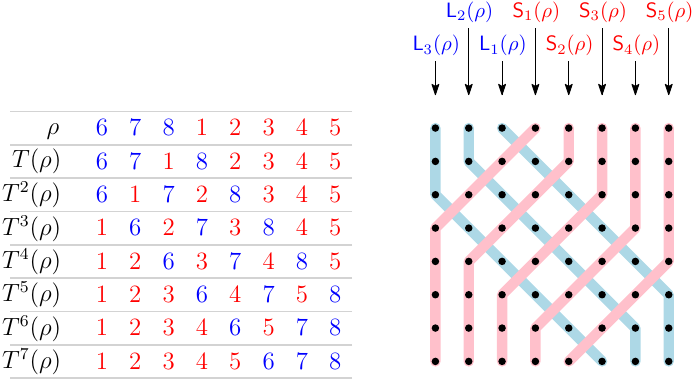}\internallinenumbers
\caption{The $k$-wiring diagram for $\rho=\ominus(n-k, k)$ for $n=8$, $k=5$.}
\label{fig:wd_rho}
\end{figure}

\begin{definition}[Poset of words]\label{def:poset}
Assume $1 \leq k \leq n-1$.
Let $\x(k,n-k)$ be the set of words of length $n$ with $k$ occurrences of $\xs$
and $n-k$ occurrences of $\xl$.
For $\lambda, \mu \in \x(k, n-k)$ we write $\lambda \preccurlyeq \mu$
if
for each $j$, $1 \leq j \leq k$,
the position of the $j$th $\xs$ in $\lambda$
is weakly to the left from
the position of the $j$th $\xs$ in $\mu$.
Equivalently,
$\lambda \preccurlyeq \mu$ if for each $j$, $1 \leq j \leq n-k$,
the position of the $j$th $\xl$ in $\lambda$
is weakly to the right from
the position of the $j$th $\xl$ in $\mu$\footnote{We remind that
the occurrences of $\xs$ are counted \textbf{from left}, and
the occurrences of $\xl$ are counted \textbf{from right}.}.
For example, $\xs\xs\xs\xl\xl \li \xs\xl\xs\xl\xs \li \xl\xs\xl\xs\xs$.
It is easy to verify that $\preccurlyeq$ is a partial order relation,
and that these two definitions are indeed equivalent.

In fact, $\preccurlyeq$ is the transitive closure of the relation
``$u$ is obtained from $v$ by replacing some consecutive occurrence of $\xl\xs$ by $\xs\xl$''.
In particular, it follows 
that $\xl^{n-k} \xs^k$ --- the maximum element of $\x(k, n-k)$ --- 
covers just one element: $\xl^{n-k-1} \xs\xl\xs^{k-1}$,
and
that $\xs^k \xl^{n-k} $ --- the minimum element of $\x(k, n-k)$ ---
is covered by just one element: $ \xs^{k-1}\xl\xs\xl^{n-k-1}$.

See Figure~\ref{fig:poset} for the Hasse diagram of
$\x(k, n-k)$ with the partial order
$\preccurlyeq$,
for $n=6, k=3$ (the words marked by blue colour are the $3$-shadows of $\rho, T(\rho), \ldots$).
\end{definition}
\begin{figure}[th]
\centering
\includegraphics[scale=0.74]{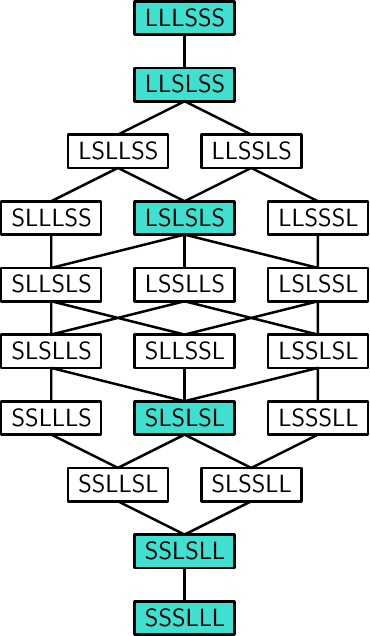}
\internallinenumbers
\caption{The Hasse diagram of $\x(k, n-k)$ with the partial order
$\preccurlyeq$, for $n=6, k=3$.
The $k$-shadows of $\rho, T(\rho),T^2(\rho),\ldots$ are shown by blue colour.}
\label{fig:poset}
\end{figure}

\begin{definition}[Bandwidth]
The \textit{bandwidth} (also called maximum displacement) 
of a permutation $\pi=a_1a_2\ldots a_n$ is
\begin{equation}
\displaystyle{\di(\pi) := \max_{1 \leq i \leq n} |a_i-i|}.
\end{equation}

The inequality $\di(\pi) \leq r$ corresponds to
the diagram of $\pi$ being $(2r+1)$-diagonal.
This means that it is entirely contained in the main diagonal and its $r$
shifts to either side.
Obviously, for each permutation $\pi$ of size $n$,
we have $\di(\pi) \leq n-1$,
and $\di(\pi) = 0$ if and only if $\pi = \mathrm{id}$.
\end{definition}

\subsection{The worst case: \texorpdfstring{$n-1$}{n-1} iterations, and a key statistic: bandwidth of permutations}\label{sec:width}

The main result of this section is the following theorem about the bandwidth of a permutation in $\imt{m}$.
In terms of permutation diagrams, this theorem says that grey areas
as in Figure~\ref{tab:ev1200} do not contain any points of respective diagrams.

\begin{theorem}\label{thm:width}
If $\tau=b_1b_2\ldots b_n \in \imtm$
(where $0 \leq m \leq n-1$),
then we have
$\di(\tau) \leq n-1-m$.
\end{theorem}

Notice that for $m=n-1$ we obtain that
a permutation in $\imt{n-1}$ has $\di = 0$, and thus we
recover Ungar's result mentioned above, $\imt{n-1}=\{\id\}$.
The proof that we give is also a generalization of
Ungar's proof of this special case\footnote{Note that another perspective 
 on Ungar's result can also be found in the recent work of Albert and Vatter~\cite{AlbertVatter2020}, 
inspired by Knuth's zero-one principle~\cite[Theorem Z of Section 5.3.4]{Knuth73},
and offering links with a nondeterministic sort.\linebreak}.

\smallskip

The proof of Theorem~\ref{thm:width} will follow from the following propositions.

\begin{proposition}\label{thm:lambda}
Let $\sigma$ be a permutation of size $n$.
Let $1 \leq k \leq n-1$.
Let $\lambda \in \x(k, n-k)$ 
(the poset defined in Definition~\ref{def:poset})
such that $\lambda \neq \xs^k \xl^{n-k}$.
Let $\lambda'$ be the word obtained from $\lambda$
by replacing each consecutive occurrence of $\xl\xs$ by
$\xs\xl$. Then we have:
If $\sh(\sigma) \li \lambda$
then $\sh(T(\sigma)) \li \lambda'$.
\end{proposition}

\begin{proof}
The word $\sh(T(\sigma))$ is obtained from $\sh(\sigma)$
by flips of the form
$\xl^a\xs^b \to \xs^b \xl^a$ (for some $a,b \geq 0$),
induced by flips in $\sigma$.
It follows that for each $j$ ($1 \leq j \leq k$)
we have $s_j(\sh(T(\sigma))) \leq s_j(\sh(\sigma))$.
Moreover, if $s_j(\sh(\sigma)) = \alpha$
and the $(\alpha-1)$st position in $\sh(\sigma))$
is $\xl$, then we have $s_j(\sh(T(\sigma))) < s_j(\sh(\sigma))$.
As for $\lambda$: since all the flips in~$\lambda$
are of the form $\xl\xs \to \xs \xl$, we always have
$s_j(\lambda') = s_j(\lambda)$
or
$s_j(\lambda') = s_j(\lambda)-1$.

In terms of $k$-wiring diagram, this means that $\xs$-paths are monotone
in the sense that, as we scan them from the top to the bottom,
they move weakly to the left at each step. Moreover, if there is $\xl$
before $\xs$, then the path that goes through this $\xs$ will
move strongly to the left at the next step.

We now prove the claim.
Assume for contradiction that we have
$\sh(\sigma) \preccurlyeq \lambda$ but
$\sh(T(\sigma)) \not \preccurlyeq \lambda'$.
Due to the properties just seen,
this can only happen if for some $j$,
we have $(\alpha:=) \ s_j (\sh(\sigma))
= s_j(\sh(T(\sigma)))
= s_j (\lambda)
= s_j (\lambda') + 1$
(see Figure~\ref{fig:posetproof}).
Since $s_j (\lambda')
= s_j (\lambda) -1$,
the $(\alpha-1)$st position in $\lambda'$
is occupied by $\xl$.
Since $s_j (\sh(T(\sigma)))
= s_j (\sh(\sigma)) $,
the $(\alpha-1)$st position in $\sh(\sigma)$
is occupied by $\xs$.
This means $s_{j-1}(\sigma)=\alpha-1$
and $s_{j-1}(\lambda)<\alpha-1$, which contradicts
$\sh(\sigma) \preccurlyeq \lambda$.
\begin{figure}[ht]
\centering
\includegraphics[scale=0.85]{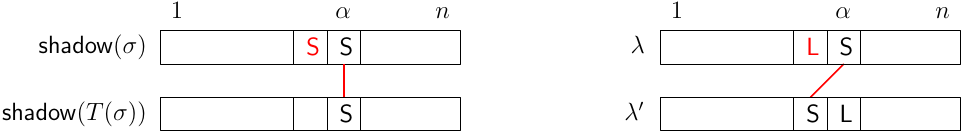}
\internallinenumbers
\caption{Illustration to the proof of Proposition~\ref{thm:lambda}
($\sh(\sigma) \preccurlyeq \lambda \ \Rightarrow \ \sh(T(\sigma)) \preccurlyeq \lambda$).}
\label{fig:posetproof}
\end{figure}
\end{proof}

\pagebreak

The next proposition is a rephrasing of one of the claims in Ungar's proof.

\begin{proposition}\label{thm:pi_rho}
Let $\pi$ be any permutation of size $n$, and let $1 \leq k \leq n-1$.
Let $\rho = \ominus(n-k, k)$ as defined in Section~\ref{sec:3def}.
Then, for each $m$ ($1 \leq m \leq n-1$), we have for $k$-shadows
$\sh(T^m(\pi)) \preccurlyeq \sh(T^m(\rho))$.
\end{proposition}

\begin{proof}
First, we have $\sh(\pi) \preccurlyeq \sh(\rho)$ because in $\sh(\rho)$
the $k$ letters $\xs$ occupy the last $k$ positions.

Then, the claim follows from Proposition~\ref{thm:lambda} by induction
since for $0 \leq m \leq n-2$ we have $\sh(T^{m+1}(\rho)) = (\sh(T^{m}(\rho)))'$.
\end{proof}

In terms of $k$-wiring diagrams, Proposition~\ref{thm:pi_rho} says that
the paths in the $k$-wiring diagram of $\rho$ \textbf{majorize}
the paths in the $k$-wiring diagram of $\pi$ in the following sense:
for each $j$ ($1 \leq j \leq k$),
the path $\xs_j (\pi)$
is weakly to the left of the path $\xs_j(\rho)$;
and similarly,
for each $j$ ($1 \leq j \leq n-k$),
the path $\xl_j (\pi)$
is weakly to the right of the path $\xl_j(\rho)$.
For illustration, see Figure~\ref{fig:wd_comb}
in which
$\xs$- and $\xl$-paths of $\pi$ (thin red and blue lines)
are shown together with
$\xs$- and $\xl$-paths of $\rho$ (thick pink and light-blue lines).

\begin{figure}[ht]
\centering\internallinenumbers
\includegraphics[scale=0.95]{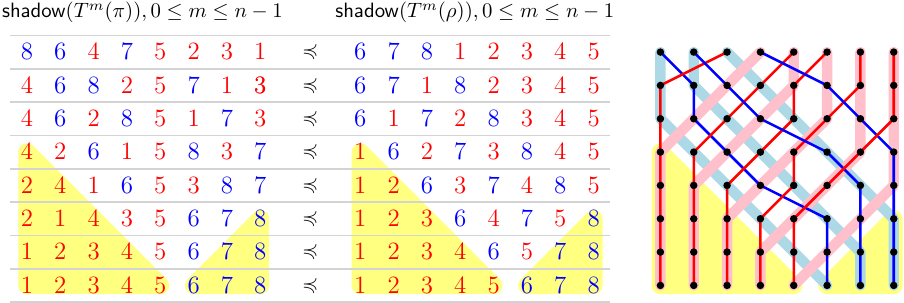}
\caption{Illustration to Proposition~\ref{thm:pi_rho}.
Left part:
For each $m\geq 0$, $\sh(T^m(\pi)) \preccurlyeq \sh(T^m(\rho))$.
Right part:
The wiring diagrams for $\pi$ and for $\rho$ shown superimposed.
For each $j$, the red path $\xs_j (\pi)$ is weakly to the left of the pink path $\xs_j(\rho)$;
for each $j$, the blue path $\xl_j (\pi)$ is weakly to the right of the light-blue path $\xl_j(\rho)$.
The triangular areas (indicated by a yellow colour)
only contain $\xs$ for the left triangle and $\xl$ for the right triangle.
}
\label{fig:wd_comb}
\end{figure}

We now can complete the proof of Theorem~\ref{thm:width} on the bandwidth of permutations in $\imtm$.

\begin{proof}[of Theorem~\ref{thm:width}]
We notice that, along the last (vertical) segment of any path of $\rho$,
the corresponding path of $\pi$ coincides with it
(these areas are marked by yellow triangles in Figure~\ref{fig:wd_comb}). More formally,
it follows from Proposition~\ref{thm:pi_rho} and from Equations~\eqref{eq:wd_rho}
(which describe the shape of the $\xs$- and $\xl$-paths of $\rho$)
that
\begin{align}
\text{for each } j \ (1\leq j \leq k)\colon \ m\geq n-k-1+j \Rightarrow s_j(\sh(T^m(\pi)))=j, \label{eq:tri01} \\
\text{for each } j \ (1\leq j \leq n-k)\colon \ m\geq k-1+j \Rightarrow \ell_j(\sh(T^m(\pi)))=n+1-j. \label{eq:tri02}
\end{align}

Now, we fix $m$ ($0\leq m \leq n-1$), and let
$\tau = T^m(\pi) = b_1 b_2 \ldots b_n$.
Since $s_j(\sh(\tau))=\alpha$ for some $j$ implies $b_\alpha \leq k$,
and since $\ell_j(\sh(\tau))=\beta$ for some $j$ implies $b_\beta \geq k+1$,
equations \eqref{eq:tri01} and \eqref{eq:tri02} translate to
\begin{align}
1 \leq i \leq k+m+1-n \ \ \Rightarrow \ \ b_i\leq k, \label{eq:tri1} \\
k+n-m \leq i \leq n \ \ \Rightarrow \ \ b_i\geq k+1. \label{eq:tri2}
\end{align}

Now we make the final step.
Let $1 \leq i \leq n$.
Equations~\eqref{eq:tri1} and~\eqref{eq:tri2}
hold for \textit{any}~$k$, $1 \leq k \leq n-1$,
and we choose two specific values.
First, we take $k=i+n-1-m$.
Now we have $i \leq k+m+1-n$, and therefore~\eqref{eq:tri1} implies $b_i \leq k = i+(n-1-m)$.
Second, we take $k=i-n+m$.
Now we have $i \geq k+n-m$, and therefore~\eqref{eq:tri2} implies $b_i \geq k+1 = i-(n-1-m)$.
Thus, we have $|b_i-i| \leq n-1-m$ for each $i$ ($1 \leq i \leq n$),
that is, $\di(\tau) \leq n-1-m$, as claimed.
\end{proof}

In terms of permutation diagrams, Equations~\eqref{eq:tri1} and~\eqref{eq:tri2}
(for fixed $m$) mean that for each $k$ some rectangular areas in the diagram of $\tau \in \imtm$ are forbidden.
Taking the union of these areas for $1 \leq k \leq n-1$, we obtain \textbf{forbidden corners},
which yield the bound on the bandwidth.
See Figure~\ref{fig:wd_zones} for illustration (the forbidden areas are indicated by grey colour).

\begin{figure}[ht]
\centering
\includegraphics[scale=.75]{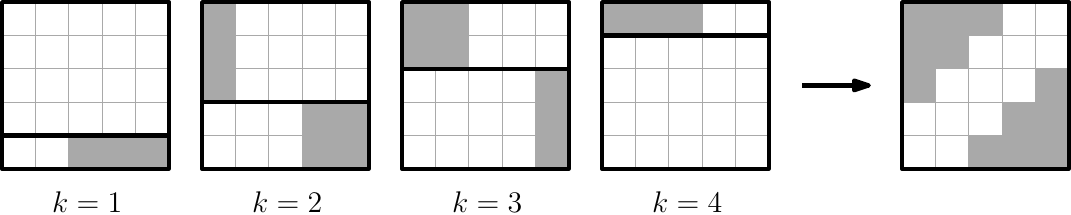}
\internallinenumbers
\caption{Illustration to proof of Theorem~\ref{thm:width} for $n=5$: forbidden areas for fixed $m$ ($m=3$).}
\label{fig:wd_zones}
\end{figure}

In Figure~\ref{fig:wd_zones_m} we show the forbidden corners for $n=5$ and $0 \leq m \leq 4$.

\begin{figure}[ht]
\centering\setlength{\abovecaptionskip}{0mm}\setlength{\belowcaptionskip}{0mm}
\includegraphics[scale=.75]{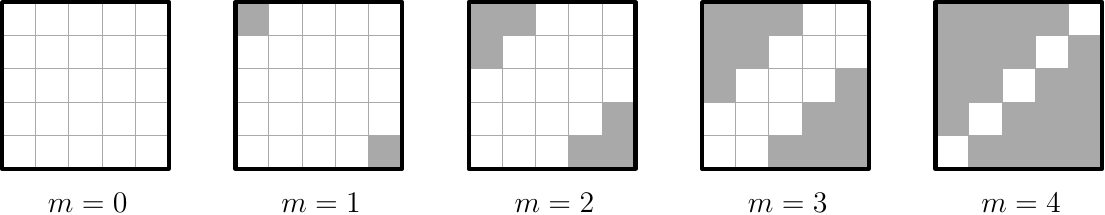}
\caption{Illustration to proof of Theorem~\ref{thm:width} for $n=5$:
forbidden areas for $0 \leq m \leq n-1$.}
\label{fig:wd_zones_m}
\end{figure}

Now, if we compare this to Figure~\ref{tab:ev1200}
(where grey areas are precisely those forbidden by Theorem~\ref{thm:width}),
we see many other areas without points, and
it is natural to ask whether there are some larger forbidden areas.
In fact, contrary to what Figure~\ref{tab:ev1200} could suggest, the forbidden areas
characterized in Theorem~\ref{thm:width} are maximal in the
sense that for each position
outside the grey corners we can find a permutation that contains a point at this position.
In fact, all the positions not forbidden by Theorem~\ref{thm:width} are ``covered'' by
skew-layered permutations with $\leq 2$ runs, as our next result shows.

\begin{proposition}\label{thm:2runs}
Let $n$ be fixed, and let $0 \leq m \leq n-1$.
For $1 \leq i \leq n$, let
$\max\{0, \, i-(n-1-m)\} \leq j \leq \min\{n, \, i+(n-1-m)\}$.
Then there exists $\pi$,
a skew-layered permutation with $\leq 2$ runs,
such that for $\tau=T^m(\pi)=b_1b_2 \ldots b_n$
we have $b_i=j$.
\end{proposition}

\begin{proof}
Assume $j<i$.
Let $\pi = \ominus(k, n-k)$ where $k$ will be determined below.
From~\eqref{eq:wd_rho} and discussion thereafter, we know:
if $m \leq j-1$ then $j=b_i$ for $i=k+j$,
if $j-1 \leq m \leq k-1+j$ then $j=b_i$ for $i=k+j-(m-(j-1))$,
and if $k-1+j \leq m \leq n-1$ then $j=b_j$.
Solving the first two expressions for $k$ we obtain
that we need to set $k=i-j$ if $m \leq j-1$,
and $k = i-2j+m+1$ if $j \leq m \leq k-1+j$.

The case $j>i$ is similar, and for $j=i$ we can take $\pi=\id$ which will do for any $m$.
\end{proof}

\begin{figure}[h]
\centering\setlength{\abovecaptionskip}{0mm}\setlength{\belowcaptionskip}{0mm}
\includegraphics[scale=.75]{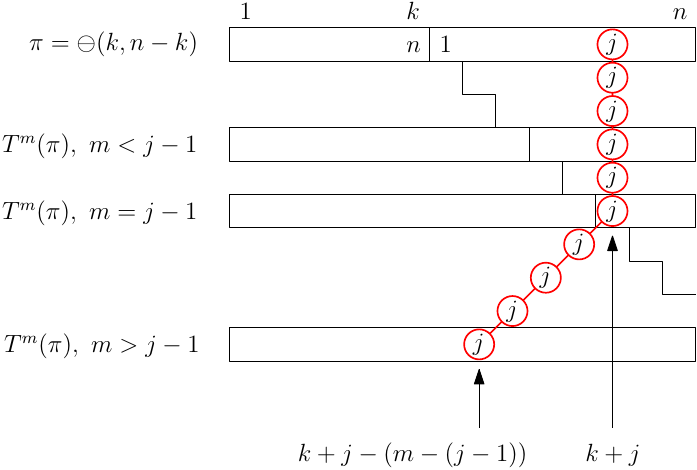}\internallinenumbers
\caption{Illustration to the proof of Proposition~\ref{thm:2runs}.}
\label{fig:2runs}
\end{figure}

Finally, we remark that, since Theorem~\ref{thm:width} establishes
that the bandwidth of allowed positions for permutations in $\imtm$ decreases with $m$,
one can be tempted to conjecture that for $\pi \neq \id$ we always have $\di(T(\pi)) < \di(\pi)$.
Such a claim would, of course, imply Theorem~\ref{thm:width}!
However, such a conjecture is wrong: the smallest counter-example is $\pi = 3412$,
and it is easy to see that any skew-layered permutation
with at least two blocks,
such that the first and last blocks are of size $\geq 2$, is a counter-example.

\pagebreak

\subsection{The penultimate step of flip-sort:
the theorem about \texorpdfstring{$\mathrm{Im}(T^{n-2})$}{Im(T(n-2))}.}

In view of Ungar's result, we have $\imt{n-1}=\{\id\}$.
The main result of this section is a structural characterization of $\imt{n-2}$.
According to Theorem~\ref{thm:width},
$\tau \in \mathrm{Im}(T^{n-2})$ implies $\di(\tau) \leq 1$.
In fact, $\di(\tau) \leq 1$ is a necessary, but not sufficient, condition for being in $\imt{n-2}$. 
(For example, for $n=5$, one can see in Figure~\ref{fig:tree} that
the permutation $21354$ with $\di=1$ is not in $\imt{3}$.)
However, permutations that satisfy $\di(\tau) \leq 1$ with will play an important role in our next results,
and so we introduce a notation for them, and list some of their (nearly obvious) properties.

\begin{definition}\label{def:thin}
A permutation $\tau$ is \textit{thin} if $\di(\tau) \leq 1$.
\end{definition}

\begin{observation}\label{thm:thin}\phantom{\qquad}
\begin{enumerate}[topsep=0pt]
\item A thin permutation is uniquely defined by the sequence of lengths of its runs.
The thin permutation with runs of lengths $r_1, r_2, \ldots, r_s$ (left to right)
will be denoted by $\th(r_1, r_2, \ldots, r_s)$.
For example, $\th(3,4,1)=124|3568|7$.
If $s\geq 2$, the runs $r_i$ with $1 < i < s$ will be referred to as \textup{inner runs}.
\item A sequence $(r_1, r_2, \ldots, r_s)$ of natural numbers
is realizable as the sequence of lengths of runs of a thin permutation
if and only if $r_i>1$ for $2 \leq i \leq s-1$.
\item Any inner run of a thin permutation $a_1 a_2 \ldots a_n \neq \id$ has the form
\begin{equation*}[a_{\ell}, a_{\ell+1}, a_{\ell+2}, \ldots, a_{m-2}, a_{m-1}, a_{m}]
=[\ell-1, \ell+1, \ell+2, \ldots, m-2, m-1, m+1]\end{equation*}
The first and, respectively, the last run, have the forms
\begin{equation*}[a_{1}, a_{2}, \ldots, a_{m-2}, a_{m-1}, a_{m}]
=[1, 2, \ldots, m-2, m-1, m+1],\end{equation*}
\begin{equation*}[a_{\ell}, a_{\ell+1}, a_{\ell+2}, \ldots, a_{n-1}, a_{n}]
=[\ell-1, \ell+1, \ell+2, \ldots, , n-1, n].\end{equation*}
\item A thin permutation is also uniquely defined by the sequence of lengths
of its falls, that can only have lengths $1$ or $2$.
Any $\{1, 2\}$-sequence is realizable as
the sequence of lengths of falls of a thin permutation.
It follows that thin permutations are enumerated by Fibonacci numbers.
\end{enumerate}
\end{observation}

Now we state and prove the characterization of $\mathrm{Im}(T^{n-2})$.

\begin{theorem}\label{thm:n-2}
A permutation $\tau=b_1b_2\ldots b_n$ belongs to $\mathrm{Im}(T^{n-2})$
if and only if it is thin and has no inner runs of odd size.

Accordingly, $|\mathrm{Im}(T^{n-2})|=2^{n/2-1}+2^{n/2}-1$ for even $n$,
$2^{(n+1)/2}-1$ for odd $n$ (\href{https://oeis.org/A052955}{OEIS A052955})\footnote{OEIS stands for the On-Line Encyclopedia of Integer Sequences, accessible at \url{https://oeis.org/}.}.
\end{theorem}

\begin{proof} \quad [\textit{First part of the proof:} $\tau \in \mathrm{Im}(T^{n-2}) \Rightarrow \tau$ is thin and with no inner runs of odd size.]
Let $\tau=b_1b_2\ldots b_n \in \mathrm{Im}(T^{n-2})$.
By Theorem~\ref{thm:width} we have $\di(\tau) \leq (n-1) - (n-2) = 1$:
thus, $\tau$ is thin.
The second condition will be proved by contradiction.
To this end we assume that $\tau$ is a thin permutation with at least one
inner run of odd size, and that we have $\tau = T^{n-2}(\pi)$
for some permutation~$\pi$.
Specifically, we assume that $\tau$
has a run $[b_{k+1}, b_{k+2}, \ldots, b_{\ell-1}, b_{\ell}]$
such that
$1 \leq k$, $\ell \leq n-1$,
$\ell-k$ is odd, and $\ell-k\geq 3$
(because, by Observation~\ref{thm:thin}.2,
a thin permutation has no inner runs of size~$1$).

For the proof, we shall exploit shadows and wiring diagrams for both $k$ and $\ell$.
First, due to Observation~\ref{thm:thin}.3,
we know the $k$- and the $\ell$-shadows of $\tau$:
\begin{equation}\label{eq:kl_shadow}
\sh_k(\tau)=\xs^{k-1}\xl\xs\xl^{n-k-1}, \ \ \ \sh_\ell(\tau)=\xs^{\ell-1}\xl\xs\xl^{n-\ell-1}.
\end{equation}

Now we introduce the \textit{$(k-\ell)$-wiring diagram of $\pi$}:
the drawing that consists of $\xs$-paths from the $k$-wiring diagram of $\pi$,
and the $\xl$-paths from the $\ell$-wiring diagram of $\pi$.
Thus, in the $(k-\ell)$-wiring diagram we have $k$ (red) $\xs$-paths:
for $1 \leq j \leq k$,
the $j$th $\xs$-path (denoted by~$\xs_j$) connects the $j$th from left occurrences of
the values from $S:=\{1, 2, \ldots, k \}$;
and we have $\ell':=n-\ell$ (blue) $\xl$-paths:
for $1 \leq j \leq \ell'$
the $j$th $\xl$-path (denoted by~$\xl_j$) connects the $j$th from right occurrences of
the values from $L:=\{\ell+1, \ell+2, \ldots, n \}$.
The remaining values from
$\{k+1, k+2, \ldots, \ell \}$
are not connected by paths in such a diagram.

Additionally, we consider the pink paths of $\rho_k=\ominus(n-k, k)$
and the light-blue paths of $\rho_\ell=\ominus(n-\ell, \ell)$.
By Proposition~\ref{thm:pi_rho}, these paths majorize the corresponding paths of $\pi$:
For $1 \leq j \leq k$, the path $\xs_j(\pi)$ is weakly to the left from the $\xs_j(\rho_k)$, and
for $1 \leq j \leq \ell'$, the path $\xl_j(\pi)$ is weakly to the right from the $\xl_j(\rho_\ell)$.

At this point we notice that, \textbf{since $\ell-k$ is odd},
all the pink paths $\xs_j(\rho_k)$
and
the light-blue paths $\xl_j(\rho_\ell)$
cross each other at grid points of the diagram,
see Figure~\ref{fig:odd}.
(Compare with Figure~\ref{fig:wd_rho},
where pink and light-blue paths cross each other between the rows of the diagram!)
Denote by $P_{i,j}$ ($1 \leq i \leq k$, $1 \leq j \leq \ell'$)
the crossing point of the path $\xs_i(\rho_k)$ and the path $\xl_j(\rho_\ell)$:
we refer to these points as \textit{conflict points}.

\begin{figure}[h]
\centering
\includegraphics[scale=0.85]{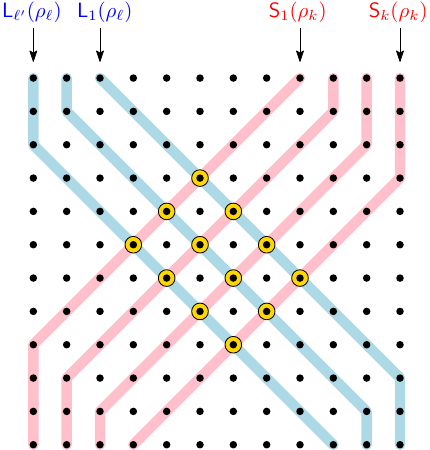}
\internallinenumbers
\caption{Illustration to the proof of Theorem~\ref{thm:n-2}:
pink paths of $\rho_k=\ominus(n-k, k)$
and blue-light paths of $\rho_\ell=\ominus(n-\ell, \ell)$.
The conflict points are drawn as yellow discs.}
\label{fig:odd}
\end{figure}

Since $S \cap L = \varnothing$, it is impossible that
both $\xs_i(\pi)$ and $\xl_j(\pi)$ pass through $P_{i,j}$.
This means that in the row of $P_{i,j}$,
at least one of the paths $\xs_i(\pi)$ and $\xl_j(\pi)$
is \textbf{strongly majorized} by the corresponding
pink or, respectively, light-blue path.

For $1 \leq i \leq k, \ 1 \leq j \leq \ell'$
we say that
the pair $(i, j)$ is \textit{regular}
if
some point of $\xs_i(\pi)$ lies on $\xs_i(\rho_k)$
after $P_{i,j}$ and
before $\xs_i(\pi)$ and $\xs_i(\rho_k)$ reach their final position $i$,
\textit{and}
some point of $\xl_j(\pi)$ lies on $\xl_j(\rho_\ell)$
after $P_{i,j}$ and
before $\xl_j(\pi)$ and $\xl_j(\rho_\ell)$ reach their final position $n+1-j$.
Next we prove two claims:
\begin{itemize}
\item \textbf{Claim 1:} \textit{The pair $(k, \ell')$ is regular.} \\
Indeed, we have $b_{k+1}=k$ and $b_{\ell}=\ell+1$;
this means that, in the row
which corresponds to $\tau=T^{n-2}(\pi)$,
$\xs_k(\pi)$ coincides with $\xs_k(\rho_k)$,
and
$\xl_{\ell'}(\pi)$ coincides with $\xl_{\ell'}(\rho_\ell)$.
In both cases, the paths are
after the conflict points and before their final columns.

\item \textbf{Claim 2:} \textit{If a pair $(i, j)$ is regular,
then at least one of the pairs $(i-1, j)$, $(i, j-1)$ is regular.} \\
At least one of the paths $\xs_i(\pi)$, $\xl_j(\pi)$ does not contain the point $P_{ij}$.
Assume without loss of generality that $\xs_i(\pi)$ does not contain the point $P_{ij}$.
That is, in the row of $P_{ij}$,
$\xs_i(\pi)$ is strictly to the left from the pink path $\xs_i(\rho_k)$.
Since $\xs_i(\pi)$ has later a point that lies on $\xs_i(\rho_k)$
(after $P_{i,j}$, but before both paths reach their final position $i$),
it makes a vertical step there, from a point $\alpha$ to a point $\beta$
(refer to Figure~\ref{fig:KLp}).
But this means that $\gamma$, the left neighbour of $\alpha$, is also from $S$.
Therefore $\gamma$ is a point of $\xs_{i-1}(\pi)$ on $\xs_{i-1}(\rho_k)$,
after $P_{i-1,j}$ but before both paths reach their final position $i-1$.
This means that the pair $(i-1, j)$ is regular.
(If we assume above that $\xl_j(\pi)$ does not contain the point $P_{ij}$,
we obtain that $(i, j-1)$ is regular.)
\end{itemize}

\begin{figure}[h]
\setlength{\abovecaptionskip}{0mm}
\setlength{\belowcaptionskip}{-1mm}
\centering
\includegraphics[scale=.98]{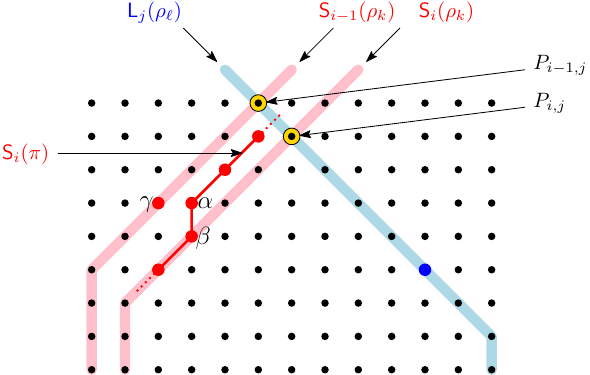}\internallinenumbers
\caption{Illustration to the proof of the claim: if $(i, j)$ is regular,
then $(i-1, j)$ is regular or $(i, j-1)$ is regular.}\label{fig:KLp}
\end{figure}

Starting with Claim~$1$ and applying Claim~$2$ repeatedly, we obtain that the pair $(1,1)$ is regular.
However, the proof of Claim~$2$ applies to $(1,1)$ as well.
Since $(1,1)$ is the highest conflict point, we get a contradiction.

\smallskip

It is also possible to reverse this proof and to explain it ``from the top to the bottom''.
Then it essentially says that at least one of the sets, $S$ and $L$,
arrives at its final position sooner than in $n-1$ iterations.
However in $\tau$ these sets are not at their final position
(as witnessed by~\eqref{eq:kl_shadow}), and this is a contradiction.

\medskip
\makebox[0.975\textwidth][s]{
[\textit{Second part of the proof:} 
$\tau$ is thin and with no inner runs of odd size $\Rightarrow \tau \in \mathrm{Im}(T^{n-2})$.]}\\
For each $\tau$ which is a thin permutation without odd inner runs, 
we will construct $\pi$ such that $T^{n-2}(\pi)=\tau$.
If $\tau=\id$ we can take $\pi = \id$, so we assume from now on that $\tau\neq\id$.
Let $\tau = \th(r_1, r_2, \ldots, r_{s-1}, r_s)$ so that $s>1$ and the numbers $r_2, \ldots, r_{s-1}$ are even.
We define $\pi$ to be the skew-layered permutation $\ominus(r_s, r_{s-1}, \ldots, r_2, r_1)$.

We claim that for this $\pi$ we have $T^{n-2}(\pi)=\tau$.
To see that, we analyse the wiring diagram of $\pi, T(\pi), T^2(\pi), \ldots$; an
example is illustrated in Figure~\ref{fig:skew2thin}.
Notice that this is the usual wiring diagram, in which paths connect fixed values.
For the sake of visualization, each inner run of $\pi$ is partitioned into two \textit{blocks}
of equal size, the left one and the right one.
For the values from the left blocks of inner runs,
and for all the values from the last run,
we colour the paths in green.
For the values from the right blocks of inner runs,
and for all the values from the first run,
we colour the paths in orange.
Then, the paths have very clear description.
Specifically, for an inner run of length $r$, the $j$th from left value ($1 \leq j \leq r/2$)
generates the green path that consists of $5$ segments (some of them can be empty):
(1) a vertical segment of length $(j-1)$,
(2) a slanted segment of slope $1$ until the path reaches the position $j$,
(3) a vertical segment of length $r-2j+1$,
(4) a slanted segment of slope $-1$ until the path reaches its final position,
(5) a vertical segment.
The green paths of the last run do not have (4) and (5);
orange paths are described similarly.
The whole diagram is symmetric, and therefore the sequence of lengths of runs
of $\tau$ is the reversal of the sequence of lengths of runs of $\pi$.

\begin{figure}[h]
\centering
\includegraphics[scale=.93]{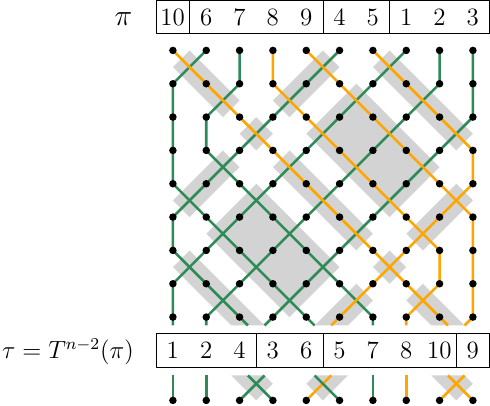}
\internallinenumbers
\caption{Illustration to the proof of $T^{n-2}(\ominus(r_s, \ldots, r_1))=\th(r_1, \ldots, r_s)$ (where $r_2, \ldots, r_{s-1}$ are even).}\label{fig:skew2thin}
\end{figure}

Our proof shows that in order to construct a permutation in $\imt{n-2}$,
one only has to choose positions (of the same parity) for the gaps between runs.
This yields the enumerative formula, as stated in the claim.
\end{proof}

\textbf{Remark.} In general, there are several permutations $\pi$ that satisfy $T^{n-2}(\pi)=\tau$
for given $\tau\in\imt{n-2}$. Characterizing all such $\pi$ is a challenging problem.

\pagebreak

\subsection{Skew-layered permutations: a family of worst cases}\label{sec:skew}

By Theorem~\ref{thm:n-2}, each skew-layered permutation
of size $n$ without odd inner runs has the 
(maximum possible for permutations of size $n$) cost $n-1$.
Characterizing all permutations with that cost is a challenging open problem. 
In this section, we present a necessary condition,
and then we conclude with a conjecture concerning the cost of any skew-layered permutation.

\begin{proposition}\label{thm:skew}
Let $\pi$ be a permutation of size $n$ with $\cost(\pi) = n-1$.
Let $\tau = b_1b_2 \ldots b_n = T^{n-2}(\pi) $.
\begin{enumerate}
\item There exists $k$, $1 \leq k \leq n-1$, such that
$\sh_k(\tau) \neq \xs^k\xl^{n-k}$.
\item For each such $k$,
we have $\sh_k(\pi) = \xl^{n-k}\xs^k$.
Consequently, $\pi$ can be written as
a skew sum $\sigma \ominus \sigma'$,
where $\sigma, \sigma'$ are permutations of sizes $n-k, k$.
\end{enumerate}
\end{proposition}

\begin{proof}
\begin{enumerate}\setlength\itemsep{0mm}
\item Since $\cost(\pi) = n-1$, we have $\tau \neq \id$.
For each $i$ with $b_i>i$ we can take $k=b_i-1$,
and for each $i$ with $b_i<i$ we can take $k=b_i$.
\item In order to find a contradiction, assume that $\sh_k(\pi) \neq \xl^{n-k}\xs^k$.
Since the only element of the set of words $\x(n-k, k)$ covered by $\xl^{n-k}\xs^k$
is $\sh_k(T(\rho)) = \xl^{n-k-1}\xs\xl\xs^{k-1}$, we necessarily have
$\sh_k(\pi) \li \sh_k(T(\rho))$.
Then, by Proposition~\ref{thm:lambda} and by induction, we have
$\sh_k(T^{n-2}(\pi)) \li \sh_k(T^{n-1}(\rho))$,
that is, $\sh_k(\tau) = \xs^k\xl^{n-k}$,
which is a contradiction. \qedhere
\end{enumerate} 
\vspace{-3mm}
\end{proof}

As for skew-layered permutations, we expect that their cost is always high
(excluding the cases of $\id$ and $-\id$); however, this cost is not always $n-1$.
Our computer experiments then suggest the following classification.

\begin{conjecture}
Let $\pi$ be a skew-layered permutation of size $n$ such that $\pi$
is neither $\id$ nor $-\id$.
\begin{itemize}
\item For even $n$, one has $\cost(\pi) = n-1$.
\item For odd $n$, if the $(n+1)/2$-st position of $\pi$ is the central point of a run or of a fall of size $\geq 3$, 
then $\cost(\pi) = n-2$. Otherwise, $\cost(\pi) = n-1$.
\end{itemize}
\end{conjecture}

As a direct corollary of this conjecture, among the $2^{n-1}$ skew-layered permutations, there are:
\begin{itemize}\setlength\itemsep{0mm}
\item for even $n$, $2^{n-1}-2$ skew-layered permutations with $\cost(\pi)=n-1$;
\item for odd $n$, $(2^{n-2} - 2) / 3$ skew-layered permutations with $\cost(\pi)=n-2$,
\item for odd $n$, $(5\cdot 2^{n-2} - 4) / 3$ skew-layered permutations with $\cost(\pi)=n-1$.
\end{itemize}
These are respectively the sequences \oeis{A000918}, \oeis{A020988}, and \oeis{A080675} from the OEIS.
\pagebreak

\section{Conclusion}
In this article, we have seen that a simple sorting procedure like flip-sort 
has many interesting combinatorial and algorithmic facets.
Our analysis of the best cases and the worst cases of this algorithm,
our results on the underlying poset, the links with lattice paths, 
are just some first steps towards a more refined understanding of this process. 

En passant, we listed a few conjectures, which, we hope, could tease the desire of the reader to work on this topic.
The optimal cost of the computation of pop-stacked permutations is still open;
other challenges are a finer analysis of the automata which recognize them (when the number of runs is fixed),
the existence of closed-form formulas, a method to solve puzzling functional equations like the one given in Theorem~\ref{thm:FE}. 

We also believe that several of the notions and structures introduced in this article 
can also be useful for the next natural step: the average-case analysis of flip-sort.
Even if this algorithm is not, in its naive version, as efficient as quicksort, 
it is very easy to implement, to parallelize, and 
it is still possible to optimize its implementation (this is a subject per se, 
like using double chained lists to perform the flips efficiently, 
or using some additional data structures to have a parsimonious reading of each iteration, etc). 
So there is not a ``flip-sort'' algorithm,
but a family of flip-sort algorithms, and to finely tune them in order to decide what could be the optimal variant 
should clearly be the subject of future research.

\bigskip

\acknowledgements
\label{sec:ack}
This research was supported by the Austrian Science Fund (FWF) 
in the framework of the project \href{https://www.math.aau.at/P28466/}{\textit{Analytic Combinatorics: Digits, Automata and Trees} (P~28466)}.
Cyril Banderier thanks the Department of Mathematics at the University of Klagenfurt for funding his visit in July 2019. 
Let us end by thanking the referees for their feedback. 
\bigskip

\bibliographystyle{cyrbiburl}
\bibliography{ABH}
\end{document}